\documentclass[10 pt, leqno]{amsart}

\usepackage{amsfonts}
\usepackage{amsthm}
\usepackage{amssymb}
\usepackage{latexsym}
\usepackage{amsmath}
\usepackage{mathrsfs}
\usepackage{comment}

\theoremstyle{plain}
\newtheorem{tw}{Theorem}[section]

\newtheorem {lem} [tw]{Lemma}

\newtheorem{thm}[tw]{Theorem}
\newtheorem{lmma}[tw]{Lemma}
\newtheorem{ppsn}[tw]{Proposition}
\newtheorem{crlre}[tw]{Corollary}
\newtheorem{xmpl}[tw]{Example}

\theoremstyle{definition}
\newtheorem {deft}[tw] {Definition}
\newtheorem {rem} [tw]{Remark}

\newcommand{\nnred}{\wt{W}_n}
\newcommand{\knred}{\wt{W}_k}
\newcommand{\lnred}{\wt{W}_l}

\newcommand{\bc} {\Bbb C}
\newcommand{\bn}{\Bbb N}
\newcommand{\br}{\Bbb R}
\newcommand{\bt}{\Bbb T}
\newcommand{\bz}{\Bbb Z}
\newcommand{\bF}{\Bbb F}

\newcommand{\Dirac}{\textup{\texttt{\large D}}}

\newcommand{\alg} {\mathsf{A}}
\newcommand{\Alg} {\mathcal{A}}
\newcommand {\Tr} {{\textup{Tr}}}
\newcommand {\Lin} {{\textup{Lin}}}
\newcommand {\LIM} {{\textup{LIM}}}
\newcommand {\card} {{\textup{card}}}
\newcommand {\QISO} {{\textup{QISO}}}
\newcommand {\qlg} {{\mathsf{Q}}}
\newcommand {\freeal} {C_r^*(\bF_2)}
\newcommand{\freeg}{\bF_2}
\newcommand {\id} {{\textup{id}}}
\newcommand{\blg} {\mathsf{B}}

\newcommand{\tu}{\textup}

\newcommand{\Com}{\Delta}

\newcommand{\Hil}{\mathsf{H}}
\newcommand{\Dom}{\textup{Dom}}

\newcommand{\Lap}{\mathcal{L}}

\newenvironment{rlist}
{

\begin{enumerate}}
{\end{enumerate}}

\newcommand{\la}{\lambda}

\newcommand{\ot}{\otimes}

\newcommand{\wt}{\widetilde}

\newcommand{\be}{\begin{equation}}
\newcommand{\ee}{\end{equation}}
\newcommand{\bea}{\begin{eqnarray}}
\newcommand{\eea}{\end{eqnarray}}
\newcommand{\bean}{\begin{eqnarray*}}
\newcommand{\eean}{\end{eqnarray*}}
\newcommand{\brray}{\begin{array}}
\newcommand{\erray}{\end{array}}
\newcommand{\ben}{\begin{equation}{nonumber}}
\newcommand{\een}{\end{equation}{nonumber}}

\newcommand{\bdfn}{\begin{dfn}}
\newcommand{\bthm}{\begin{thm}}
\newcommand{\blmma}{\begin{lmma}}
\newcommand{\bppsn}{\begin{ppsn}}
\newcommand{\bcrlre}{\begin{crlre}}
\newcommand{\bxmpl}{\begin{xmpl}}
\newcommand{\brmrk}{\begin{rmrk}}
\newcommand{\edfn}{\end{dfn}}
\newcommand{\ethm}{\end{thm}}
\newcommand{\elmma}{\end{lmma}}
\newcommand{\eppsn}{\end{ppsn}}
\newcommand{\ecrlre}{\end{crlre}}
\newcommand{\exmpl}{\end{xmpl}}
\newcommand{\ermrk}{\end{rmrk}}

\numberwithin{equation}{section}


\begin{document}

\author{Jyotishman Bhowmick}
\footnote{\emph{Permanent address of the second named author:} Faculty of Mathematics and Computer Science, University of
\L\'{o}d\'{z}, ul. Banacha 22, 90-238 \L\'{o}d\'{z}, Poland.}
\address{Stat-Math Unit, Indian Statistical Institute, 203, B. T. Road, Kolkata 700 208} \email{jyotishmanb@gmail.com}
\author{Adam Skalski}
\address{Department of Mathematics and Statistics,  Lancaster University,
Lancaster, LA1 4YF  }
  \email{a.skalski@lancaster.ac.uk}

\title{\bf Quantum isometry groups of noncommutative manifolds associated to group $C^*$-algebras}

\begin{abstract}
Let $\Gamma$ be a finitely generated discrete group. The standard spectral triple on the group $C^*$-algebra $C^*(\Gamma)$ is
shown to admit the quantum group of orientation preserving isometries. This leads to new examples of compact quantum groups. In
particular the quantum isometry group of the $C^*$-algebra of the free group on $n$-generators is computed and turns out to be a
quantum group extension of the quantum permutation group $A_{2n}$ of Wang. The quantum groups of orientation and real structure
preserving isometries are also considered and construction of the Laplacian for the standard spectral triple on $C^*(\Gamma)$
discussed.
\end{abstract}

\maketitle

From the early days of the theory of spectral triples a prominent role was played by an example of a spectral triple on a group
$C^*$-algebra $C^*(\Gamma)$, where $\Gamma$ is a (discrete) group equipped with a fixed length function $l:\Gamma \to \br_+$ (it
appears already in Connes' original paper \cite{co}).  When $\Gamma$ is finitely generated, properties of the natural word-length
function on $\Gamma$ and its associated spectral triple reflect deep combinatorial geometric aspects of $\Gamma$ -- so for
example summability of the triple corresponds to the growth conditions, Rieffel regularity of the spectral triple is closely
related to the Haagerup's Rapid Decay property (\cite{OR}).

Nowadays spectral triples are viewed as objects best suited to describe noncommutative (compact, Riemannian) manifolds. Thus,
especially in the mathematical physics literature, the focus seems to be on investigating spectral triples on deformations of
function algebras of classical spaces.   In recent work of Goswami and the first named author (\cite{Deb}, \cite{JyotDeb2},
\cite{Deb_real}), the notion of a quantum isometry group of a spectral triple, understood as the universal compact quantum group
acting on the corresponding noncommutative manifold in the way compatible with the spectral triple structure was introduced. It
was shown that if the spectral triple is sufficiently well-behaved, the quantum isometry group exists, and in the classical
situation the definition of `compatible' actions given in \cite{Deb} describes precisely the relevant group actions preserving
the Riemannian metric. Moreover Goswami and the first named author investigated quantum isometry groups of classical manifolds
and their noncommutative deformations. In \cite{ours} together with Goswami we studied quantum isometry groups of spectral
triples on $AF$ $C^*$-algebras introduced in \cite{Chrivan}. Already there the nature of the problem becomes more combinatorial
and many connections with quantum permutation groups (\cite{Wang}, \cite{Teosurvey}), or more general  quantum symmetry group of
finite graphs (\cite{graph}) appear.

The aim of this paper is the initiation of the study of quantum isometry groups of spectral triples on group $C^*$-algebras. We
begin by introducing basic notations and recalling the construction of spectral triples on group $C^*$-algebras associated to a
length function on the group. In Section 2 we describe the approach to the quantum group of orientation preserving isometries due
to Goswami and the first named author and show that the triples described in Section 1 fit in the framework studied in
\cite{JyotDeb2}.  The first concrete examples appear in Section 3 and come from singly generated abelian groups. In the case
$\Gamma= \bz$, $C^*(\Gamma) = C(\bt)$ and the spectral triple associated to the usual symmetric generator set $\{-1,1\}$ is the
classical triple associated to the differentiation, so that the quantum isometry groups coincides with the classical one, as was
shown in \cite{JyotDeb1}. Direct computations for $\Gamma=\bz/_{n\bz}$ $(n \geq 5)$ show that here the resulting quantum isometry
groups are also commutative, but, curiously, the one associated with $\bz/_{4\bz}$ is not (a similar phenomenon was earlier
observed for quantum symmetry groups of $n$-gons in \cite{BanJFA}).

The next two sections, 4 and 5, are devoted to the thorough analysis of two more complicated examples -- the smallest
noncommutative group $S_3$, both in its permutation and dihedral incarnations (interpreted as leading to two different generating
sets), and the queen of all noncommutative discrete groups, free group on $n$-generators ($n \geq 2$). The free group case leads
to new compact quantum groups which are generated by partial isometries and can be viewed as extensions of quantum permutation
groups of Wang mentioned above. Section 6 is devoted to the analysis of properties of a natural real structure of the spectral
triples in question and corresponding quantum groups of isometries preserving the real structure.  In the last section we discuss
the construction of corresponding Laplacians and the relation between the associated  quantum isometry groups in the sense of
\cite{Deb} and these studied in Sections 3-5.

\section{Basic notations and the standard construction and properties  of spectral triples on group $C^*$-algebras} \label{introLap}

The spatial (minimal) tensor product of $C^*$-algebras will be denoted by $\ot$ and the algebraic tensor product of two algebras will be denoted by $ \ot_{{\rm alg}} .$ If $\alg$ is a $C^*$-algebra, $S(\alg)$ will
denote the state space of $\alg$. The cyclic groups $\bz/_{n \bz}$ will be denoted simply by $\bz_n$ ($n \in \bn)$.

For the definition of a spectral triple we refer to \cite{co} or \cite{Conti}. Throughout this article, we will work with
`compact' noncommutative manifolds, i.e.\ assume that Dirac operators have compact resolvents. As we are only interested here in
odd spectral triples, the grading will not play any role.

Let $\Gamma$ be a discrete group equipped with a length function $l:\Gamma \to \br_+$. Recall that $l$ has to satisfy the
following conditions:
\[ l(g) = 0 \textrm{ iff } g=e, \;\; l(g^{-1}) = l(g), \;\; l(gh) \leq l(g) + l(h), \;\; g,h \in \Gamma.\]
We will  assume that the length takes only integer values and that for each $n \in \bn$ the set $W_n:=\{g \in \Gamma:
l(g) = n\}$ is finite.

The most important example is given by the word length induced by a fixed finite symmetric set of generators in  (finitely
generated) $\Gamma$. For most purposes we will assume that the chosen symmetric set of generators is \emph{minimal}.  Define the
operator $\Dirac_{\Gamma}$ on $l^2(\Gamma)$ by
\[ \Dom (\Dirac_{\Gamma}) = \{\xi \in l^2(\Gamma): \sum_{g \in \Gamma} l(g)^2 |\xi(g)|^2 < \infty\},\]
\[ (\Dirac_{\Gamma} (\xi)) (g) = l(g) \xi(g), \;\;\; \xi \in \Dom(\Dirac_{\Gamma}), g \in \Gamma.\]

If $g $ belongs to $ \Gamma, $ we will denote by $ \delta_{g} $ the function in $ l_2 ( \Gamma ) $ which takes the value $ 1 $ at
the point $ g $ and $ 0 $ at all other points. The natural generators of the algebra $ \bc[\Gamma] $ (and their images in the
left regular representation) will be denoted by $\lambda_g$.

Consider the left regular representation of $\Gamma$ on $l_2(\Gamma)$ and extend it to a representation of either $C^*(\Gamma)$
or $C^*_r(\Gamma)$. It is easy to check that $(\bc[\Gamma], l^2(\Gamma), \Dirac_{\Gamma})$ is a spectral triple (note that we
will often view $\bc[\Gamma]$ as a subalgebra of $B(l^2(\Gamma))$). It can be viewed as a spectral triple on the full group
$C^*$-algebra $C^*(\Gamma)$, but as all our considerations will concern the left regular representation, we will usually view it
as a triple on $C^*_r(\Gamma)$. For more information on such spectral triples, extension to weighted Dirac operators and related
topics we refer to a recent paper \cite{Conti}. Regularity properties of triples of that type were considered for example in
\cite{Rief} and \cite{OR}. If $l$ is the usual word-length function associated to a finite symmetric set of generators $S$, it is
easy to see that $\card (W_n)\leq \card(S)^n$, so that the triple $(\bc[\Gamma], l^2(\Gamma), \Dirac_{\Gamma})$ is automatically
\emph{$\theta$-summable} - for each $t>0$ the (bounded) operator $\exp(-t\Dirac^2)$ is trace class. This is confirmed by a simple
calculation:
\[ \Tr(\exp(-t\Dirac^2)) = \sum_{\gamma\in \Gamma} \textup{e}^{-tl(\gamma)^2} = \sum_{n=0}^{\infty}\card(W_n) e^{-tn^2} <
\infty.\]

Let $\tau$ denote the canonical tracial state on $C^*_r(\Gamma)$, that is the vector state associated to the (cyclic and
separating) vector $\delta_e \in l^2(\Gamma)$. It is easy to see that the eigenvectors of $\Dirac_{\Gamma}$ belong to the (dense)
subspace $\bc[\Gamma] \delta_e$. Moreover we also have $\bc[\Gamma] = \Lin \{a \in \bc[\Gamma]: a \delta_e  \textrm{ is an
eigenvector of } \Dirac_{\Gamma}\}$. These observations will be used in the following section.

\section{Quantum group of orientation preserving isometries associated to spectral triples on $C^*(\Gamma)$ -- general theory} \label{genth}

The following notion of quantum families of orientation preserving isometries was introduced in \cite{JyotDeb2}:

\begin{deft} \label{family}
A quantum family of orientation preserving isometries for the spectral triple $(\Alg, \Hil, \Dirac)$ is given by a pair $(
{\mathcal S} ,U)$, where $  {\mathcal S} $ is a separable unital $C^*$-algebra and $U$ is a linear map from $ \Hil \rightarrow
\Hil \ot   {\mathcal S}  $ such that $ \wt{U} $ given by  $ \wt{U} ( \xi \ot b ) = U ( \xi ) ( 1 \ot b ), ~ ( \xi \in \Hil, ~ b
\in   {\mathcal S}  ) $ extends to a unitary element of the multiplier algebra $M(K(\Hil) \ot   {\mathcal S} )$ satisfying the
following conditions:
\begin{rlist}
\item for every state $\phi \in   {\mathcal S} ^*$ the operator $U_{\phi} = (\id_{K(\Hil)} \ot \phi) (\wt{U}) \in B(\Hil)$ commutes with $\Dirac$;
\item if we define $\alpha_U ( x ) = \wt{U}(x \ot 1) {\wt{U}}^* $ for $x \in B(\Hil)$, then for every state $\phi \in   {\mathcal S} ^*$
and each $a \in \Alg$ the operator $  (\id_{K(\Hil)} \ot \phi) \alpha_U ( a ) $ belongs to $\Alg''$.
\end{rlist}

In case the $ C^* $-algebra ${\mathcal S}$ has a coproduct $\Delta $ such that $ (   {\mathcal S} , \Delta ) $ is a compact
quantum group and $ U $ is a unitary corepresentation of $ (   {\mathcal S} , \Delta ) $ on $ \Hil, $ we say that $ (   {\mathcal
S} , \Delta ) $ acts by orientation preserving isometries on the spectral triple $(\Alg, \Hil, \Dirac)$.
\end{deft}


Consider the category $\mathbf{ Q} (\Alg, \Hil, \Dirac)$ with the object-class consisting of all quantum families of orientation
preserving isometries $ (   {\mathcal S} , U ) $ of the given spectral triple, and the set of morphisms $ {\rm Mor} ( ( {\mathcal
S} , U ), (   {\mathcal S} ^{\prime}, U^{\prime} ) ) $ being the set of unital $ C^* $-homomorphisms $ \Phi: {\mathcal S}
\rightarrow   {\mathcal S} ^{\prime} $ satisfying $ ( \id \ot \Phi ) ( U ) = U^{\prime}.$ We also consider another category
$\mathbf{ Q^{\prime}} (\Alg, \Hil, \Dirac)$ whose objects are triplets $ (   {\mathcal S} , \Delta, U ), $ where $ ( {\mathcal S}
, \Delta ) $ is a compact quantum group acting by orientation preserving isometries on the given spectral triple, with $ U $
being the corresponding unitary corepresentation. The morphisms are the morphisms in the category of compact quantum groups which
are also morphisms of the underlying quantum families of orientation preserving isometries. For further details we refer to
Section 2 of \cite{JyotDeb2}. The universal object in the category $\mathbf{\mathbf Q} (\Alg, \Hil, \Dirac)$ (if it exists) will
be denoted by $\wt{\QISO^{+}}(\Alg, \Hil, \Dirac).$

The terminology `orientation preserving' in this context was introduced in \cite{JyotDeb2} and is motivated by the fact that if
we consider the spectral triple of the type $(C^{\infty}(\tu{M}), \Hil, \Dirac)$, where $\tu{M}$ is a compact Riemannian spin
manifold, $\Hil$ is the $L^2$-space of spinors and $\Dirac$ is the Dirac operator acting on the spinors, then the category
described above corresponds precisely to the category of families of orientation preserving isometries acting on $\tu{M}$.


Theorem 2.23 of \cite{JyotDeb2} states the following:

\begin{tw}\label{existence}
Let $(\Alg, \Hil, \Dirac)$ be a spectral triple of compact type. Assume that $\Dirac$ has a one-dimensional kernel spanned by a
vector $\xi \in \Hil$ which is cyclic and separating for $\Alg$ and further that each eigenvector of $\Dirac$ belongs to $\Alg
\xi$. Then the the universal object in the category $\mathbf{Q} (\Alg, \Hil, \Dirac)$ exists. Moreover $ \wt{\QISO^+}(\Alg, \Hil,
\Dirac) $ admits a coproduct making it a compact quantum group and the universal object in the category $\mathbf{Q^{\prime}}
(\Alg, \Hil, \Dirac)$.

If we denote by ${\QISO}^+(\Alg, \Hil, \Dirac)$ the Woronowicz $ C^* $-subalgebra of $\wt{{\QISO}^+}(\Alg, \Hil, \Dirac)$
generated by elements of the form $ \left\langle \alpha_{U_0} ( a ) ( \eta \ot 1 ), \eta^{\prime} \ot 1
\right\rangle_{\wt{{\QISO}^+}(\Alg, \Hil, \Dirac)} $ where $ \eta, ~ \eta^{\prime} \in \Hil$, $ a \in  \Alg $ and  $ \left\langle
\cdot , \cdot \right\rangle_{\wt{{\QISO}^+}(\Alg, \Hil, \Dirac)} $ denotes the $ \wt{{\QISO}^+}(\Alg, \Hil, \Dirac) $ valued
inner product of $ \Hil \ot \wt{{\QISO}^+}(\Alg, \Hil, \Dirac), $ then $\wt{{\QISO}^+}(\Alg, \Hil, \Dirac)$ is isomorphic (as a
compact quantum group) to the free product ${\QISO}^+(\Alg, \Hil, \Dirac)  \star \,C(\bt).$ The compact quantum group $ {\QISO}^+
( \Alg, \Hil, \Dirac ) $ will be called the quantum group of orientation preserving isometries of $(\Alg, \Hil, \Dirac)$.
\end{tw}


\begin{rem}

As $ U_0 $ commutes with $ \Dirac, $ we have
\begin{equation} U_0 ( \xi ) = \xi \ot q \label{U_0_xi} \end{equation}
 for a unitary element $ q $ in $ \wt{{\QISO}^+}(\Alg, \Hil, \Dirac).$ It can be seen from the proof of Theorem 2.23 of \cite{JyotDeb2} that $ C^* ( q ) \cong C ( \bt ) $ and this $ C ( \bt )  $ appears in the expression for $ \wt{{\QISO}^+}(\Alg, \Hil, \Dirac)$ as given in  Theorem \ref{existence}.

\end{rem}


Let $\Gamma$ be a finitely generated discrete group with a fixed finite symmetric set of generators $S$ and let $l$ be the
word-length function on $\Gamma$ associated with $S$. Let $(\bc[\Gamma], l^2(\Gamma), \Dirac_{\Gamma})$ denote the spectral
triple described in the previous section. It is easy to check that the conditions in Theorem \ref{existence} are satisfied, so
that the quantum group of orientation preserving isometries ${\QISO}^+(\bc[\Gamma], l^2(\Gamma), \Dirac_{\Gamma})$ exists. We
will further denote it by ${\QISO}^+(\Gamma, S)$ to stress the fact that it may depend on the initial choice of generating set.
Note that the quantum group ${\QISO}^+(\Gamma, S)$ actually acts on the `dual' object of $\Gamma$ in the sense of the quantum
group theory.
 In
concrete examples ${\QISO}^+(\Gamma, S)$ can be computed directly using Definition \ref{family}; the condition (ii) in that
definition takes a suitably simple form, which is described by the next remark, which can be deduced easily from the discussion
before Proposition 2.5.4 in \cite{BOzawa}.

\begin{rem}
Let $\qlg$ be a (unital) $C^*$-algebra, $\Hil = l^2(\Gamma)$ and $\alg = C^*_r(\Gamma)$. Let
\[X= \left\{T \in M(\qlg \ot K(\Hil)): \forall_{\phi \in S(\qlg)} (\phi \ot \id_{K(\Hil)})(T) \in \alg''\right\}.\]
Then $X=\left\{T \in M (\qlg \ot K(\Hil)): \forall_{g,h \in \Gamma} (\id_{\qlg} \ot \omega_{\delta_e, \delta_g}) (T) =
(\id_{\qlg} \ot \omega_{\delta_h, \delta_{hg}})(T)\right\},$ \\where $\omega_{\xi, \eta} (a) := \langle \xi, a \eta\rangle$ for
all $\xi, \eta \in \Hil$, $a \in K(\Hil)$.
\end{rem}

In our context we can however offer an even more direct way of computing ${\QISO}^+(\Gamma, S)$. Recall that $\tau$ denotes the
faithful trace on $C^*_r(\Gamma)$ given by the vector state associated to $\delta_e \in l^2(\Gamma)$.

Let $ ( \Alg, \Hil, \Dirac ) $ be a spectral triple satisfying the conditions of Theorem \ref{existence}. Let $ {\mathcal A}_{00}
= {\rm Lin} \{ a \in \Alg : a \xi ~ {\rm is ~ an ~ eigenvector ~ of} ~ \Dirac \}.  $ Moreover, assume that $ {\mathcal A}_{00} $
is norm dense in $ \Alg. $

Let $ \hat{\Dirac}: {\mathcal A}_{00} \rightarrow {\mathcal A}_{00} $ be defined by $ \hat{\Dirac} ( a ) \xi = \Dirac ( a \xi ) $
$(a \in \Alg_{00}$). This is well defined as $ \xi $ is a cyclic and separating vector for $ \Alg, $ so also for $\Alg_{00}$.

\begin{deft} \label{C_hat}
Let $\alg$ be a $ C^*$-algebra and $ \Alg $ be a dense $ \ast $-subalgebra such that $ ( \Alg, \Hil, \Dirac ) $ is a spectral
triple as above. Let $  {\bf \widehat{C}} ( \Alg, \Hil, \Dirac ) $ be the category with objects $ ( \qlg, \alpha ) $ such that $
\qlg $ is a compact quantum group with a $ C^*$-action
 $ \alpha $ on $\alg$ such that
\begin{rlist}
\item $ \alpha $ is $ \tau $ preserving, that is, $ ( \tau \ot \id ) \alpha ( a ) = \tau ( a )1_{\qlg} $ for all $ a $ in $ \alg $;
\item $ \alpha $ maps $ {\mathcal A}_{00} $ into $ {\mathcal A}_{00} \ot_{{\rm alg}} \qlg $;
\item  $ \alpha \hat{\Dirac} = ( \hat{\Dirac} \ot I ) \alpha .$
\end{rlist}
The morphisms in $  {\bf \widehat{C}} ( \Alg, \Hil, \Dirac ) $  are compact quantum group morphisms intertwining the respective
actions.
\end{deft}

Under the assumptions stated before the last definition we have the following result (Corollary 2.27 of \cite{JyotDeb2}):

\begin{tw}\label{D_hat}
There exists a universal object $ \widehat{\qlg} $ in $ {\bf \widehat{C}} (\Alg, \Hil, \Dirac ) $. It is isomorphic to $
{\QISO}^+ (\Alg, \Hil, \Dirac ). $
\end{tw}

Thus to compute ${\QISO}^+(\Gamma, S)$ we need to consider an action $\alpha$ of ${\QISO}^+(\Gamma, S)$ on $C^*(\Gamma)$ such that there exist elements
$\{q_{\gamma', \gamma}\in {\QISO}^+ ( \Gamma, S ): \gamma,\gamma' \in \Gamma\}$ such that
\begin{equation} \alpha(\la_{\gamma}) = \sum_{\gamma' \in \Gamma:l(\gamma)=l(\gamma')} \la_{\gamma'} \ot q_{\gamma',\gamma}, \;\;\; \gamma \in \Gamma.\label{alphag}\end{equation}
The condition $ ( \alpha ( \lambda_{\gamma} ) )^*  = \alpha ( {\lambda_{\gamma}}^* )  = \alpha ( \lambda_{\gamma^{- 1}} ) $
implies that $ \sum_{\gamma' \in \Gamma:l(\gamma') = l(\gamma)} \la_{\gamma'^{- 1}} \ot q^*_{\gamma',\gamma} = \sum_{\gamma' \in
S:l(\gamma') = l(\gamma^{-1})} \la_{\gamma'} \ot q_{\gamma',\gamma^{- 1}} $. This leads to the condition
\begin{equation}  q_{\gamma^{\prime - 1},\gamma} = q^*_{\gamma^{\prime},\gamma^{-1}}, \;\;\; \gamma, \gamma' \in \Gamma, \; l(\gamma)=l(\gamma').\label{alpha_star_hom_cond}\end{equation}
 Passing to the generators and using the fact that $\alpha$ preserves the trace $\tau$ we observe that
 the matrix $[q_{t,s}]_{t,s \in S}$ is a unitary in $M_{\textup{card}(S)}({\QISO}^+ ( \Gamma, S))$. This is a consequence of the
 string of equalities ($s,t \in S$, we write $\qlg:={\QISO}^+(\Gamma, S)$ and use the Dirac delta notation):
 \begin{align*}
\delta_{s,t} \ot 1_{\qlg}&= \delta_{s^{-1} t, e} \ot 1_{\qlg} = \tau(\la_{s^{-1}t}) \ot
1_{\qlg} = (\tau \ot \id_{\qlg}) (\alpha(\la_{s^{-1}t}))\\
&= (\tau \ot \id_{\qlg}) (\alpha(\la_s)^* \alpha(\la_t)) = (\tau \ot \id) \left(\left(\sum_{s' \in S} \la_{s'} \ot q_{s',
s}\right)^* \left(\sum_{t' \in S} \la_{t'} \ot q_{t', t}\right)\right) \\
&= \sum_{s', t' \in S} \tau (\la_{s'^{-1} t'}) q_{s', s}^* q_{t', t} = \sum_{s' \in S} q_{s', s}^* q_{s', t}
 \end{align*}
and the analogous computation starting from $\tau(\la_{st^{-1}})$ (arguments of that type can be also found for example in
Section 2 of \cite{Soltan}).
 In fact ${\QISO}^+(\Gamma, S)$ is the universal $C^*$-algebra generated by the symbols $\{q_{s,t}: s,t \in S\}$ satisfying the relations
making $\alpha$ defined by \eqref{alphag} above extend, in a necessarily unique way, to a unital $^*$-homomorphism
$C^*_r(\Gamma)\to C^*_r(\Gamma) \ot {\QISO}^+ ( \Gamma, S )$ (the elements $q_{\gamma, \gamma'}$ appearing in \eqref{alphag} can
be defined inductively in terms of $q_{t,s}$). The fact that the map $\alpha$ can be considered on the level of reduced group
$C^*$-algebras, and not only the universal ones, follows from its $\tau$-preserving property. Note that the matrix
$[q_{t,s}]_{t,s \in S}$ is a fundamental unitary corepresentation of ${\QISO}^+(\Gamma, S)$. As stated above, $ \wt{{\QISO}^+} (
\Gamma, S )$ is the free product of ${\QISO}^+(\Gamma, S)$ with $C(\bt)$. If $ q $ is as in \eqref{U_0_xi},
 the unitary $\wt{U} \in M(\wt{{\QISO}^+}(\Gamma, S) \ot K(\Hil))$ inducing the action of ${\QISO}^+(\Gamma, S)$ is determined  by the conditions
\begin{equation} \label{unitaction1} \wt{U} (\delta_{e} \ot 1_{\wt{{\QISO}^+}(\Gamma,
S)}) = \delta_e \ot q, \ee \be  \label{unitaction2} \wt{U} (\delta_{\gamma} \ot 1_{\wt{{\QISO}^+}(\Gamma, S)}) =
\sum_{\gamma^{\prime} \in \Gamma, \, l(\gamma^{\prime})=l(\gamma)}  \delta_{\gamma^{\prime}} \ot  q_{\gamma^{\prime},\gamma} q
,\;\; \gamma \in \Gamma,\end{equation} where for a $C^*$-algebra $ \blg, $ we used the identification of $M(\blg \ot K(\Hil))$
with the $C^*$-algebra of all adjointable operators on a $\blg$-Hilbert module $\Hil \ot \blg.$

\section{Singly generated groups}

A straightforward calculation based on the approach described in the last section shows that ${\QISO}^+(\bz_2,\{1\})$ is isomorphic (as a compact quantum group) to $C^*(\bz_2)$ (so to $\bc^2$ as a
$C^*$-algebra). The next simplest possible cases fitting
in the framework described in the previous section are those of $\Gamma=\bz_n$, $n \geq 3$ and $\Gamma = \bz$. We compute here
the resulting  quantum groups of orientation  preserving isometries.

\subsection*{The case of $\bz_n$}
Consider $\Gamma=\bz_n$ with the standard symmetric generating set $S=\{1,n-1\}$ ($n \geq 3$). We then obtain the following:

\begin{tw} \label{thZn}
Let $n\in \bn \setminus\{1,2,4\}$. The quantum group of orientation  preserving  isometries $\QISO(\bz_n, S)$ is isomorphic  to $C^* (\bz_n) \oplus C^*(\bz_n)$ as a $C^*$-algebra. Its action on $C^*(\bz_n)$ is given by the formula
\[
\alpha ( \lambda_{1}) = \lambda_{1} \otimes A + \lambda_{{n-1}} \otimes B,\] where $A$ and $B$ are respectively identified with $
\lambda_{1} \oplus 0, 0 \oplus \lambda_{1}\in C^* (\bz_n) \oplus C^*(\bz_n)$.  The coproduct of ${\QISO}^+(\bz_n, S)$ can be read
out from the condition that $\left(\begin {array} {ccccc} A   &  B  \\ B^* & A^* \end {array} \right)$ is the fundamental
corepresentation.
\end{tw}

\begin{proof}
The action of  ${\QISO}^+(\bz_n, S)$ on $C^*(\bz_n)$ is determined by the formula \be \label{groupalgebras_Zn_0a} \alpha (
\lambda_{1} ) = \lambda_{1} \otimes A + \lambda_{n - 1} \otimes B, \ee where $A, B$ are some elements in ${\QISO}^+(\bz_n,S)$.
Due to the fact that $\lambda_{n-1} = \lambda_1^*$ and $\alpha$ is $^*$-preserving, we must have \be \label{groupalgebras_Zn_0b}
\alpha ( \lambda_{n-1} ) = \lambda_{1} \otimes B^* + \lambda_{n - 1} \otimes A^*. \ee It follows from the discussion after
Theorem \ref{D_hat} that  the matrix $\left(\begin {array} {ccccc} A & B
\\ B^* & A^*
\end {array} \right)$ is the fundamental corepresentation of ${\QISO}^+(\bz_n, S)$ (so in particular it is a unitary in $M_2
({\QISO}^+(\bz_n, S))$). Thus $AB+BA=0$ and both $A$ and $B$ are normal. We have then \be \label{groupalgebras_Zn_0c} \alpha (
\lambda_2) = \lambda_{2} \otimes A^2 + \lambda_{n-2} \ot B^2.\ee

 Let $n=3$. Then the comparison of \eqref{groupalgebras_Zn_0b} and \eqref{groupalgebras_Zn_0c} yields
\[A^2 = A^*, B^2=B^*,\]
so that the unitarity of the fundamental corepresentation yields $A^3+B^3=1$. Further
\[\alpha(\lambda_0) = \alpha(\lambda_1) \alpha(\lambda_2)= \la_0 \ot (AA^* +BB^*) + \la_1 \ot BA^* +  \lambda_{2} \otimes AB^*.\]
As $ \alpha $ commutes with $ \hat{\Dirac}, $ this implies $AB^* = B A^*  =0.$ We claim that $A^3$ is a projection. Indeed,
\[(A^3)^2 = A^3 (1 - B^3) = A^3 - A^3 B^* B = A^3.\]
Moreover, $ {( A^3 )}^* = {( A^2 )}^3 = A^3. $
Similarly we show that $B^3$ is a projection. Finally, $AB= A {( B^2 )}^* = 0.$ Hence, $ BA = - AB =0,$  so that
$A$ and $B$ generate `orthogonal' copies of $C^*(\bz_3)$.

Let then $n>4$. Then by formulae \eqref{groupalgebras_Zn_0a} and \eqref{groupalgebras_Zn_0c} we have
\[ \alpha (\lambda_3) =  \alpha (\lambda_2) \alpha(\lambda_1)= \lambda_3 \otimes A^3 + \lambda_{n+1} \ot A^2B
+  \lambda_{n-1} \otimes B^2A + \lambda_{n-3} \ot B^3.\]
This means that $A^2B = B^2A=0$. As a consequence, for all $k=1,2,..., n-1$ we have
\be \label{induction} \alpha(\lambda_k) = \lambda_k \ot A^k + \lambda_{n-k} \ot B^{k}. \ee
Indeed, the cases $k=1,2$ are covered by formulas \eqref{groupalgebras_Zn_0a} and \eqref{groupalgebras_Zn_0c} and the inductive reasoning for $k \geq 2$ gives
\begin{align*} \alpha(\la_{k+1}) &= \alpha(\la_k) \alpha(\la_1) = ( \lambda_k \ot A^k + \lambda_{n-k} \ot B^{k})
(\lambda_{1} \otimes A + \lambda_{n - 1} \otimes B) \\&= \lambda_{k+1} \ot A^{k+1} + \lambda_{n-k-1} \ot B^{k+1},\end{align*} so
that \eqref{induction} follows.  Combining it (in the case of $k =n-1$) with \eqref{groupalgebras_Zn_0b} we see that $A^{n-1} =
A^*$, $B^{n-1}=B^*$ and further $A^*B= A^{n-1} B = A^{n-3}A^2B = 0$. This together with the equation $ A^* B + B A^* = 0 $ which
follows from the unitarity condition for the fundamental corepresentation implies that $BA^*=0.$ Hence, $ A^* $ commutes with $ B
$ and as $A$ and $B$ are normal, $ A^* $ commutes with $ B^* $ and $AB+BA=0$ implies that $AB = BA = 0.$ Thus again $A$ and $B$
are `partial unitaries' satisfying the conditions $A^{n-1} = A^*$, $B^{n-1}=B^*$ with orthogonal ranges summing to $1$. The
theorem follows.

\end{proof}

Theorem \ref{thZn}  specifically excluded $n=4$. Curiously, the quantum group of orientation preserving isometries ${\QISO}^+(\bz_4,S)$ (for $S=\{1,3\}$) is a noncommutative $C^*$-algebra. As mentioned in the introduction, similar behaviour
appears when one considers quantum symmetry groups of $n$-gons (see \cite{BanJFA} and references therein) -- the quantum symmetry
groups are commutative for all $n\in \bn$ apart from $n=4$.

\begin{tw} \label{thZ4}
The quantum orientation  preserving  isometry group ${\QISO}^+(\bz_4,S)$  is as a $C^*$-algebra the universal $C^*$-algebra
generated by two normal elements $A,B$ satisfying the following relations:
\[ AB+BA = AB^*+BA^* = A^*B+ BA^*=0, \;A^2+B^2=(A^*)^2 + (B^*)^2, \] \[ A^2 B + B^3 = B^*,  \; B^2 A + A^3 = A^*, A^4 +B^4 +2A^2B^2=1, AA^*+BB^*=1.\]
 The  action of ${\QISO}^+(\bz_4,S)$ on $C^*(\bz_4)$ is given by the formula
\[
\alpha ( \lambda_{1}) = \lambda_{1} \otimes A + \lambda_{3} \otimes B.\]  The  coproduct of ${\QISO}^+(\bz_4,S)$  can be read out
from the condition that $\left(\begin {array} {ccccc} A   &  B  \\ B^* & A^* \end {array} \right)$ is the fundamental
corepresentation. The $C^*$-algebra ${\QISO}^+(\bz_4,S)$ is not commutative.
\end{tw}

\begin{proof}
The proof follows as in Theorem \ref{thZn}, we leave the details to the reader. For the last statement it suffices to exhibit two
anticommuting selfadjoint operators whose squares sum to 1 and which do not commute. One such example is given by $A=\frac{1}{\sqrt{2}}
\sigma_1$, $B=\frac{1}{\sqrt{2}} \sigma_2$ (where $\sigma_1, \sigma_2$ are Pauli matrices).
\end{proof}

Already in this simple case we can see that if one chooses a non-minimal generating set in $\Gamma$, the resulting quantum
orientation  preserving  isometry group will be different. In particular if we put  $S'=\{1,2,3\}$, the quantum group
${\QISO}^+(\bz_4,S')$ is not isomorphic to the one obtained in the above theorem. This can be shown by analysing the quotients of
the ${\QISO}^+(\bz_4,S)$ and ${\QISO}^+(\bz_4,S')$ by respective commutator ideals and checking that in the first case one
obtains the algebra $C(\tu{Sym}_2 \times \tu{Sym}_2)$ and in the second the algebra $C(\tu{Sym}_4)$.

\subsection*{The case of $\bz$}

Consider now $\Gamma=\bz$ with the standard symmetric generating set $S=\{1,-1\}$. The next theorem describes the quantum
isometry group ${\QISO}^+(\bz, S)$.

\begin{tw}\label{bzS}
The quantum orientation  preserving  isometry group  ${\QISO}^+(\bz, S)$ is isomorphic (as a compact quantum group) to the
(commutative) compact quantum group $C(\bt \rtimes \bz_2)$. Its action on $C^*(\bz)\cong C(\bt)$ is given  by the standard
(isometric) action of the group $\bt \rtimes \bz_2$ on $\bt$.
\end{tw}

\begin{proof}
Let  the action of ${\QISO}^+(\bz, S)$ on $C^*(\bz)$ be determined by the formula
\[ \alpha ( \lambda_{1} ) = \lambda_{1} \otimes A + \lambda_{- 1 } \otimes B, \]
where $A,B \in {\QISO}^+(\bz, S)$. Via the canonical identification of $C^*(\bz)$ with $C(\bt)$ the action can be written as \be
\label{actCT}\alpha (z) = z \otimes A + \overline{z} \otimes B.\ee
 One can see that the conditions for the action on $C^*(\bz)$
to commute with the operator $ \hat{\Dirac} $ constructed from the length function are exactly the same as the conditions for the
action on $C(\bt)$ to commute with the standard Laplacian on $C(\bt)$. Thus from the proof of Theorem 2.4 of \cite{JyotDeb1} we
obtain that ${\QISO}^+(\bz, S)$ is a quantum subgroup of the classical isometry group of $ \bt $ and hence is a commutative
$C^*$-algebra. More specifically, $A=UP=PU$, $B=UP^{\perp}$ where $U$ is a unitary and $P$ is a projection, $UP=PU$. Thus $
\alpha ( z^n ) = \lambda_{1} \otimes A^n + \lambda_{- 1 } \otimes B^n $ for all $ n \geq 1 $ implying that $ C ( {\rm ISO} ( \bt
) ) $ is a sub-object of ${\QISO}^+(\bz, S)$ in the category ${\bf \widehat{C}}$ (see Definition \ref{C_hat}). Hence,
${\QISO}^+(\bz, S)$ has to coincide with the standard isometry group of the classical manifold $\bt$.
\end{proof}

It might seem surprising that the  ${\QISO}^+(\bz, S)$ admits also isometries of $\bt$ which do not preserve the
orientation in the classical sense. This is caused by the fact that the Dirac operator coming from the length function on $\bz$
has (in the $L^2(\bt)$-picture) the spectral decomposition of the type $\sum_{n \in \bz} |n| z^n$ and its eigenspaces coincide
 with these of the classical Laplacian, and not the usual Dirac operator on $L^2(\bt)$.

When $\Gamma=\bz^2$ the standard choice of generators leads to the Dirac operator on $C^*(\bz^2) \cong C(\bt^2)$ whose
eigenspaces are different from those arising from the natural Laplacian on $L^2(\bt^2)$.  We were not able to determine whether
the resulting quantum group of orientation preserving isometries $\QISO^+(\bz^2, \{(1,0), (0,1)\})$ is commutative or not.

\section{The case of $\Gamma=S_3$ with different sets of generators} \label{S3Section}

In this section we compute the quantum isometry group associated to the smallest noncommutative group, the group of permutations
of the three-element set. We will compute ${\QISO}^+(S_3, S) $ for two different set of generators; the first of them is natural if we
view $S_3$ as a permutation group generated by its transpositions, whereas the second is natural when one thinks of $S_3$ as a dihedral
group.

\subsection*{Generating set of transpositions}
We start with the set of generators $S=\{s,t\}$, where $s,t$ are two different transpositions in $S_3$, say $s = (1,2),t = (2, 3)$. Then $ s $ and $ t $ satisfy the relations $ s^2 = t^2 = e $ and $ tst = sts,$ where $e$ as usual denotes the identity
element of the group.

Let $ \alpha $ be the action of ${\QISO}^+(S_3, S)$  on $ C^*(S_3)$. As explained in the section above, there exist elements $A,
B, C, D $ in $ {\QISO}^+(S_3, S)$ such that
\begin{equation} \label{S3action} \alpha ( \lambda_{s} ) = \lambda_{s} \otimes A + \lambda_{t} \otimes B,
\;\;\; \alpha ( \lambda_{t} ) = \lambda_{s} \otimes C + \lambda_{t} \otimes D. \end{equation}

Now we derive some relations between $ A, B, C, D $ which follow from the fact that $ \alpha $ is a $ \ast $-homomorphism and
that it commutes with the operator $ \hat{\Dirac} $ associated to the generating set $S.$

\blmma \label{form1} Let $A,B,C,D$ be the elements of ${\QISO}^+(S_3, S)$ determined by the formula \eqref{S3action} for the action of
${\QISO}^+(S_3, S)$ on $C^*(S_3)$. Then the following hold:
  \be \label{groupalgebra_S_3_firstset1} A^2 + B^2 = 1, \ee
  \be \label{groupalgebra_S_3_firstset2} AB = 0, \ee
  \be \label{groupalgebra_S_3_firstset3} BA = 0, \ee
  \be \label{groupalgebra_S_3_firstset4} C^2 + D^2 = 1, \ee
  \be \label{groupalgebra_S_3_firstset5} CD = 0, \ee
  \be \label{groupalgebra_S_3_firstset6} DC = 0, \ee
  \be \label{groupalgebra_S_3_firstset7} AC + BD = 0, \ee
  \be \label{groupalgebra_S_3_firstset8} CA + DB = 0, \ee
  \be \label{groupalgebra_S_3_firstset9} DAC = CBD = 0, \ee
  \be \label{groupalgebra_S_3_firstset10} ADB = BCA = 0, \ee
  \be \label{groupalgebra_S_3_firstset11} DAD + CBC = ADA + BCB, \ee
  \be \label{groupalgebra_S_3_firstset12} A^* = A, B^* = B, C^* = C, D^* = D. \ee
 \elmma

\begin{proof}
We derive \eqref{groupalgebra_S_3_firstset1} - \eqref{groupalgebra_S_3_firstset3}  from $ s^2 = e$,
\eqref{groupalgebra_S_3_firstset4}- \eqref{groupalgebra_S_3_firstset6} from $t^2 = e$, \eqref{groupalgebra_S_3_firstset7} and
\eqref{groupalgebra_S_3_firstset8} by equating the coefficient of $ \lambda_{e} $ in $ \alpha ( \lambda_{s} \lambda_{t} ) $ and $
\alpha ( \lambda_{t} \lambda_{s} ) $ to zero, \eqref{groupalgebra_S_3_firstset9} and \eqref{groupalgebra_S_3_firstset10} by
equating coefficients of $ \lambda_{s} $ and $ \lambda_{t} $ in $ \alpha ( \lambda_{tst} ) $ and $ \alpha ( \lambda_{sts} ) $ to
zero, \eqref{groupalgebra_S_3_firstset11} from $\alpha ( \lambda_{tst} ) = \alpha ( \lambda_{sts} ) $ and finally
\eqref{groupalgebra_S_3_firstset12} from the facts that $ \alpha ( {\lambda_{s}}^* ) = ( \alpha ( \lambda_{s} ) )^*, ~ \alpha (
{\lambda_{t}}^* ) = ( \alpha ( \lambda_{t} ) )^* .$

\end{proof}

 \bthm \label{qisoS31}
 The  quantum group of orientation preserving isometries ${\QISO}^+(S_3, S) $ for the generating set built of transpositions
($S=\{s,t\}$) is isomorphic to $C^* ( S_3 ) \oplus C^*(S_3)$ as a $C^*-$algebra. Its action on $C^*(S_3)$ is given by the formula
\eqref{S3action}, where $ A, B, C, D $ are respectively identified with  $ \lambda_s \oplus 0, 0 \oplus \lambda_t,   0 \oplus
\lambda_s, \lambda_t \oplus 0 \in C^* (S_3) \oplus C^* (S_3)$.  The coproduct of ${\QISO}^+(S_3, S)$ can be read out from the
condition that $ \left(  \begin {array} {cccc} A   &  B  \\ C & D \end {array} \right ) $ is the fundamental corepresentation.
\ethm

\begin{proof}
We use the notations of Lemma \ref{form1}. The commutation relations listed in that lemma  imply that $ {( A^2 )}^2 = A A^2 A = A
( 1 - B^2 ) A = A^2 - A B B A = A^2.$ Thus, as $A^2$ is self-adjoint, $A^2$ is a projection, to be denoted by $P$. Similarly,
$B^2$ is a projection.  As $ A^2 + B^2 = 1, $ we have $ B^2 = P^{\bot}$. Proceeding in the same way,  we obtain $ C^2 = Q $ and
$D^2 = Q^{\bot}$ for another projection $Q$.

Multiplying \eqref{groupalgebra_S_3_firstset7} by $A $ on the left and $C$ on the right, we obtain $ PQ = 0. $ Hence, $ Q \leq
P^{\bot}. $ Similarly, multiplying the same formula by  $B$ on the left and $D$ on the right yields $ P^{\bot} Q^{\bot} = 0 $
which implies $ Q^{\bot} \leq P$. This implies that $P^{\bot} = Q$. Thus $A^2 = D^2 = P, \,B^2 = C^2 = P^{\bot}.$

Next, from \eqref{groupalgebra_S_3_firstset11} we deduce that $ DAD - ADA = BCB - CBC$. As $ \tu{Ran} ( DAD - ADA
) \subseteq \tu{Ran} ( P )$, $ \tu {Ran} (BCB - CBC) \subseteq \tu{Ran} P^{\bot},$ this implies
 \be \label{groupalgebra_S_3_firstset13} DAD = ADA, \ee
 \be \label{groupalgebra_S_3_firstset14} BCB = CBC. \ee

Now, we note that, by virtue of Theorem \ref{D_hat}, the equations \eqref{groupalgebra_S_3_firstset1} --
\eqref{groupalgebra_S_3_firstset12} together with the condition that the matrix $ \left(\begin {array} {cccc} A   &  B  \\ C & D
\end {array} \right)$ is a unitary  form all the conditions necessary to ensure that $C^* \{A, B, C, D\}$ is an object of $ {\bf
\widehat{C}} ( S_3, S ) $ (recall Definition \ref{C_hat}). Recall that it follows from the discussion after Theorem\ref{D_hat} the matrix $ \left(\begin {array} {cccc} A   &  B  \\
C & D \end {array} \right)$ is a fundamental unitary corepresentation of $ {\QISO}^+(S_3,S)$.

The action of the antipode $\kappa$ of a compact quantum group on the matrix elements of a finite-dimensional  unitary
corepresentation $U^\beta \equiv ( u_{pq}^\beta)$ is given by $\kappa (u_{pq}^\beta ) =( u_{qp}^\beta )^* $ (see
\cite{VanDaele}).  Thus in our situation we have $ \kappa ( A ) = A^* = A ,\,  \kappa (B) = C^* = C, ~ \kappa ( C ) = B, ~ \kappa
( D ) = D^* = D.$

Applying the antipode to the equations \eqref{groupalgebra_S_3_firstset2}  and \eqref{groupalgebra_S_3_firstset5} and then taking
adjoints,  we arrive at the equations $AC = 0$ and $BD = 0$. Thus also $DAC = CBD = 0$ and  the conditions in
\eqref{groupalgebra_S_3_firstset9} follow. Similarly we deduce the equalities \eqref{groupalgebra_S_3_firstset7},
\eqref{groupalgebra_S_3_firstset8} and \eqref{groupalgebra_S_3_firstset10}. Thus, $C^* \{A, B, C, D\} $ is the universal
$C^*$-algebra generated by elements $A,B,C,D$ such that the matrix $ \left(  \begin {array} {cccc} A   &  B  \\ C & D \end
{array} \right ) $ is a unitary and the relations \eqref{groupalgebra_S_3_firstset1}-\eqref{groupalgebra_S_3_firstset6} along
with \eqref{groupalgebra_S_3_firstset12}-\eqref{groupalgebra_S_3_firstset14} hold. Then it is easy to see that $ C^* \{ A, B, C,
D \} \cong C^* ( S_3 ) \oplus C^* ( S_3 ) $ via the map sending $ A, D, C, B $ to $ \lambda_s \oplus 0, ~ \lambda_t \oplus 0, ~ 0
\oplus \lambda_s, ~ 0 \oplus \lambda_t$. The rest of the statements are easy consequences of the proof above and the general
discussion in Section 2.
\end{proof}


\subsection*{Generating set of a transposition and a cycle}
This time we consider a set of generators of $S_3$ given by a transposition and a cycle: $S'=\{t',s'\}$, where say $ s^{\prime} = (1 2)$,  $t^{\prime} = (1 2 3)$ . The generating elements satisfy the conditions $ s^{\prime 2} = t^{\prime 3} = e, ~
t^{\prime}s^{\prime} = s^{\prime} t^{\prime - 1}.$

Let $ \alpha $ be the action of ${\QISO}^+(S_3, S')$  on $ C^*(S_3)$. This time there exist elements $E,
F,G,H,K,L $ in $ {\QISO}^+(S_3, S')$ such that
\begin{equation} \label{2S3action}
\alpha (\lambda_{t^{\prime}}) = \lambda_{t^{\prime}} \otimes E + \lambda_{t^{\prime - 1}} \otimes F + \lambda_{s^{\prime}}
\otimes G, \;\;\; \alpha ( \lambda_{s^{\prime}} ) = \lambda_{t^{\prime}} \otimes H + \lambda_{t^{\prime - 1}} \otimes K +
\lambda_{s^{\prime}} \otimes L. \end{equation} The following series of lemmas list the conditions that have to be satisfied by
$E,F,G,H,K,L \in {\QISO}^+(S_3, S')$. \blmma \label{S3lem1}
$$ HK + KH + L^2 = 1, $$
$$ H^2 = 0, $$
$$ K^2 = 0, $$
$$ HL + LK = 0, $$
$$ KL + LH = 0.$$
\elmma
\begin{proof} These follow from the relation $s^{\prime 2}=e.$
\end{proof}

\blmma \label{S3lem2}
$$ EF + FE + G^2 = 0, $$
$$ FG + GE = 0, $$
$$ EG + GF = 0 $$
\elmma
\begin{proof}
The above follow from the fact that the coefficients of $e, s^{\prime}t^{\prime}$ and  $s^{\prime} t^{\prime - 1} $ in $ \alpha (t^{\prime 2})= \alpha( t^{\prime - 1})$ are zero.
\end{proof}

\blmma \label{S3lem3}
$$ E^2 F + ( EG + GF ) G = 0, $$
$$ E^2 G + ( EG + GF ) F = 0, $$
$$ F^2 E + ( FG + GE ) G = 0, $$
$$ F^2 G + ( FG + GE ) E = 0, $$
$$ ( FG + GE ) F + ( EG + GF ) E = 0, $$
 $$ E^3 + F^3 = 1.$$
\elmma
\begin{proof}
The first five equations follow from the fact that the coefficients of $t^{\prime}, s^{\prime}t^{\prime}, t^{\prime 2}$, $
s^{\prime} t^{\prime - 1}, s^{\prime}$ in $ \alpha(t^{\prime 3})$ are zero. The last equation is a consequence of the identity
$\alpha(t^{\prime 3}) = e \otimes 1.$
\end{proof}

\blmma \label{S3lem4}
$$ EK + FH + GL = HE^2 + K F^2 = 0, $$
$$ FK = K E^2 + L ( FG + GE ) = 0, $$
$$ H ( FG + GE ) + K ( EG + GF ) = 0, $$
$$ EH = H F^2 + L ( E G + G F ), $$
$$ EL + GK = K ( FG + GE ) + L E^2, $$
$$ FL + GH = H ( EG + GF ) + L F^2. $$
\elmma
\begin{proof}
Here we use the fact that $t^{\prime}s^{\prime} = s^{\prime} t^{\prime - 1}$. As the coefficients of
$\lambda_{e},\lambda_{t^{\prime}},\lambda_{s^{\prime}}$ in $\alpha(\lambda_{t^{\prime}s^{\prime}}) = \alpha (\lambda_{s^{\prime} t^{\prime - 1}})$ are zero, we have the first three equations. The rest follows by comparing the coefficients of $ \lambda_{t^{\prime 2}}, \lambda_{s^{\prime} t^{\prime - 1}}, \lambda_{s^{\prime} t^{\prime}}$ once again in the equality $\alpha(\lambda_{t^{\prime}s^{\prime}}) = \alpha (\lambda_{s^{\prime} t^{\prime - 1}})$.
\end{proof}

\blmma \label{S3lem5}
$$ HF + KE + LG = 0, $$
$$ KF = 0. $$
\elmma
\begin{proof}
These follow by equating the coefficients corresponding to  $e$ and  $t$ in $\alpha(\lambda_{s^{\prime}t^{\prime}})$ to zero.
\end{proof}

\blmma \label{S3lem6}
$$ H = K^*, ~ L = L^*, E^2 = E^*, ~ F^2 = F^*, ~ G^* = 0 , $$
$$ FG + GE = 0, ~ EG + GF = 0.$$
\elmma
\begin{proof}
These follow from $\alpha({\lambda_{s^{\prime}}}*) = (\alpha(\lambda_{s^{\prime}}))^* $ and $\alpha({\lambda_{t^{\prime}}}*) = \alpha(\lambda_{t^{\prime 2}}) = (\alpha (\lambda_{t^{\prime}}))^* .$
\end{proof}

We are now ready to formulate the counterpart of Theorem \ref{qisoS31}.

\begin{tw} \label{qisoS32}
   ${\QISO}^+(S_3, S') $ for the generating set built of a transposition and of a cycle
($S'=\{s',t'\}$) is isomorphic to $C^* ( S_3 ) \oplus C^*(S_3)$ as a $ C^*$-algebra. Its action on $C^*(S_3)$ is given by the
formula
\[
\alpha ( \lambda_{t^{\prime}} ) = \lambda_{t^{\prime}} \otimes E + \lambda_{t^{\prime - 1}} \otimes F,  \;\;\; \alpha(\lambda_{s^{\prime}})= \lambda_{s^{\prime}} \otimes L,\]
where $E, F, L$ are respectively identified with  $ \lambda_{t'} \oplus 0, 0 \oplus \lambda_{t'},  \lambda_{s'} \oplus
\lambda_{s'} \in C^* ( S_3 ) \oplus C^* ( S_3 )$.  The coproduct of ${\QISO}^+(S_3, S')$ can be read out from the
condition that \\$\left(\begin {array} {ccccccc} E   &  F  & 0\\ F^2 & E^2 & 0 \\ 0 & 0 & L\end {array} \right)$ is the fundamental corepresentation.
\end{tw}
\begin{proof}
As the general way of arguing to this is identical to that in the proof of Theorem \ref{qisoS31}, we only provide the sketch of
the calculations involved. Let $E,F,G,H,K,L \in {\QISO}^+(S_3, S')$ be as in the lemmas preceding the theorem. Lemma \ref{S3lem6}
implies that $G = 0$, so that we have $\alpha(\lambda_{t^{\prime- 1}}) = \lambda_{t^{\prime}} \otimes F^2 + \lambda_{t^{\prime-
1}} \otimes E^2.$ Using the ordered vector space basis $\delta_{t^{\prime}}, \delta_{t^{\prime - 1}}, \delta_{s^{\prime}}$ we see
that the matrix $\left( \begin {array} {cccc}E   &  F & 0  \\ F^2 & E^2 & 0 \\ H & K & L \end {array} \right )$ is a unitary
corepresentation of ${\QISO}^+(S_3,S')$. Applying the antipode $\kappa$ we deduce that $ H = K = 0$. Using the above statements
together with Lemmas \ref{S3lem1}--\ref{S3lem6} we are left with the following set of conditions that have to be satisfied by
$E,F$ and $L$:
\[ L^2 = 1, \; EF + FE = 0, \; E^2 F = 0, \; F^2 E = 0, \; E^3 + F^3 = 1, \]
\[ EL = L E^2, \; FL = L F^2, \; L = L^*, \; E^2 = E^*, \; F^2 = F^*, \]
plus the set of relations following from the requirement that the matrix $\left(  \begin {array} {cccc}
     E   &  F & 0  \\ F^2 & E^2 & 0 \\ 0 & 0 & L \end {array} \right ) $ is a unitary. The latter statement translates to the fact that $ \left(  \begin {array} {cccc}
     E   &  F  \\ F^* & E^* \end {array} \right ) $ is a unitary and that $L$ is a unitary.
Thus the final set of equations we obtain is the following:
\[  EE^* + F F^* = 1, \;EF + FE = 0, \;F^* F + E^* E = 1, \; E^* E + F F^* = 1, \;E^* F + F E^* = 0, \]
\[ F^* E + E F^* = 0, \;
 F^*
F + E E^* = 1, \;E^2 F = 0,\; F^2 E = 0, \;E^3 + F^3 = 1,\]
\[E L = L E^2, \;F L = L F^2, \; L = L^*, \;E^2 = E^*, \;F^2 = F^*,
\;L^* L = L L^* = 1. \]
 From the above it is easy to see that the map from ${\QISO}^+(S_3, S')$ to $ C^* ( S_3 ) \oplus C^*( S_3 ) $ defined by sending $ L, E, F $ respectively to $
\lambda_{s^{\prime}} \oplus \lambda_{s^{\prime}}, ~ \lambda_{t^{\prime}} \oplus 0, ~ 0 \oplus \lambda_{t^{\prime}} $
is an isomorphism.
\end{proof}


It follows from Theorems \ref{qisoS31} and \ref{qisoS32} that the canonical actions of ${\QISO}^+(S_3,S)$ and ${\QISO}^+(S_3,S')$
on $C^*(S_3)$are different; moreover the formulas for coproduct with respect to the given $C^*$-algebraic isomorphisms
${\QISO}^+(S_3,S)\cong C^*(S_3) \oplus C^*(S_3)\cong {\QISO}^+(S_3,S')$ also differ. More specifically, if we write $s=(1 2)$, $t
= (2 3)$, $t'=(1 2 3)$, and denote for brevity $(\lambda_g \oplus 0)\in C^*(S_3) \oplus C^*(S_3)$ by $\hat{\lambda}_g$ and $(0
\oplus \lambda_g)$ by $\tilde{\la}_g$, the coproduct in the first case is given by the ($^*$-homomorphic extension) of the
formulas
\[\Com_1 (\hat{\la}_s) = \hat{\la}_s \ot \hat{\la}_s + \tilde{\la}_t \ot \tilde{\la}_s,\]
\[ \Com_1 (\hat{\la}_t) = \hat{\la}_t \ot \hat{\la}_t + \tilde{\la}_s \ot \tilde{\la}_t,\]
\[ \Com_1 (\tilde{\la}_s) = \tilde{\la}_s \ot \hat{\la}_s + \hat{\la}_t \ot \tilde{\la}_s,\]
\[ \Com_1 (\tilde{\la}_t) = \tilde{\la}_t \ot \hat{\la}_t + \hat{\la}_s \ot \tilde{\la}_t,\]
and in the second case by the ($^*$-homomorphic extension) of the formulas
\[ \Com_2 (\hat{\la}_{t'}) = \hat{\la}_{t'} \ot \hat{\la}_{t'} + \tilde{\la}_{t'} \ot \tilde{\la}_{t'^2},\]
\[ \Com_2 (\tilde{\la}_{t'}) = \hat{\la}_{t'} \ot \tilde{\la}_{t'} + \tilde{\la}_{t'} \ot \hat{\la}_{t'^2},\]
\[ \Com_2 (\hat{\la}_s + \tilde{\la}_{s} ) = (\hat{\la}_s + \tilde{\la}_{s} ) \ot (\hat{\la}_s + \tilde{\la}_{s} ).\]
Now the computation of values of the above coproducts on for example $\hat{\la}_s + \tilde{\la}_{s}$ confirms the statement from
the beginning of this paragraph.

\section{Free group on $2$ generators} \label{freegroupqiso}

Let $\bF_2$ denote the free group on two generators, let $s,t \in \bF_2$ be the generators and let $l$ be the word-length
function induced by $S=\{s,s^{-1}, t, t^{-1}\}$. The resulting Dirac operator on $l^2(\Gamma)$ will be denoted simply by
$\Dirac$.


Let $\qlg$ denote the compact quantum group ${\QISO}^+ (\bc[\bF_2],  l^2(\freeg), \Dirac),$ whose existence is guaranteed by the
arguments in Section \ref{genth}. Let $\alpha: \freeal \to \freeal \ot \qlg$ denote the canonical action of $\qlg$ on $\freeal$.
Again by the arguments of Section \ref{genth} the action $\alpha$ must `preserve' the span of $\{\la_s, \la_{s^{-1}}, \la_t,
\la_{t^{-1}}\}$, so that we must have (taking into account the selfadjointness of $\alpha$)
\[ \alpha(\la_s) = \la_s \ot A + \la_{s^{-1}} \ot B +  \la_t \ot C + \la_{t^{-1}} \ot D,\]
\[ \alpha(\la_{s^{-1}}) = \la_s \ot B^* + \la_{s^{-1}} \ot A^* +  \la_t \ot D^* + \la_{t^{-1}} \ot C^*,\]
\[ \alpha(\la_t) = \la_s \ot E + \la_{s^{-1}} \ot F +  \la_t \ot G + \la_{t^{-1}} \ot H,\]
\[ \alpha(\la_t^{-1}) = \la_s \ot F^* + \la_{s^{-1}} \ot E^* +  \la_t \ot H^* + \la_{t^{-1}} \ot G^*,\]
where $A,B,C,D,E,F,G,H$ are certain elements of $\qlg$.
Moreover, the matrix
\[ U =\begin{bmatrix}
                           A & B & C & D \\
                           B^* & A^* & D^* & C^* \\
                           E & F & G & H \\
                           F^* & E^* & H^* & G^* \\
                         \end{bmatrix}
                         \]
is a unitary in $M_4(\qlg)$. Note that

\[ U^* =\begin{bmatrix}
                           A^* & B & E^* & F \\
                           B^* & A & F^* & E \\
                           C^* & D & G^* & H \\
                           D^* & C & H^* & G \\
\end{bmatrix}.
                         \]
The additional relations (apart from these following from the unitarity conditions $U^* U = I = UU^*$) are given by the fact that
\begin{equation}\alpha(\la_s) \alpha(\la_{s^{-1}}) = \alpha(\la_{s^{-1}}) \alpha(\la_{s}) = \alpha(\la_t) \alpha(\la_{t^{-1}}) =
\alpha(\la_{t^{-1}}) \alpha(\la_{t}) = I_{\freeal \ot \qlg}= \la_e \ot I_{\qlg}.\label{alphafree}\end{equation} The new relations
arising in this way are of the form:
\begin{equation} AB^* = AC^* = AD^*= BC^* = BD^* = CD^*= A^*B = A^*C = A^*D = B^*C =B^* D = C^*D=0,\label{e1}\end{equation}
\begin{equation} EF^* = EG^* = EH^* = FG^* = FH^* = GH^*= E^*F = E^*G = E^*H = F^*G =F^* H = G^*H=0.\label{e2}\end{equation}
As a consequence of the unitarity of $U$ we have in particular:
\begin{equation} B^*B + A^*A + C^*C +D^*D=I_{\qlg},\label{e3}\end{equation}
\begin{equation} F^*F + E^*E + G^*G +H^*H=I_{\qlg}.\label{e4}\end{equation}
Multiplying \eqref{e3} on the left by respectively $A,B,C$ or $D$ and using \eqref{e1} we obtain that each of the operators
$A,B,C,D$ is a partial isometry ($AA^*A = A, BB^*B=B$, etc.). Similarly using \eqref{e2} and \eqref{e4} we obtain that $E,F,G,H$
are partial isometries. Denote by $P_A$ the range projection of $A$ (so that $P_A= AA^*$) and by $Q_A$ the initial projection of
$A$ (so that $Q_A = A^*A$) and introduce the analogous notation for  range and initial projections of $B,C,D,E,F,G$ and $H$. The
`diagonal' parts of the unitarity conditions for $U$ read in this language as
\begin{equation} P_A + P_B + P_C + P_D =  Q_A + Q_B + Q_C + Q_D =  P_E + P_F + P_G + P_H =  Q_E + Q_F + Q_G + Q_H =I_{\qlg}  \label{diag1}\end{equation}
\begin{equation} Q_A + P_B +Q_E + P_F = Q_B + P_A + Q_F +P_E = Q_C + P_D + Q_G + P_H = Q_D + P_C + Q_H + P_G= I_{\qlg}.  \label{diag2}\end{equation}
Note that \eqref{diag1} and \eqref{diag2} are equivalent to the fact that
\[ \begin{bmatrix}
                           P_A & P_B & P_C & P_D \\
                           P_E & P_F & P_G & P_H \\
                           Q_B & Q_A & Q_D & Q_C \\
                           Q_F & Q_E & Q_H & Q_G \\
                         \end{bmatrix}
                         \]
is a magic unitary in the sense of Wang (see \cite{Wang}, \cite{Teosurvey}). This leads to the following theorem:

\begin{tw} \label{mainfree}
 The quantum orientation preserving isometry group of the spectral triple $(\bc[\bF_2], l^2(\freeg), \Dirac)$, denoted by $\qlg:={\QISO}^+(\freeg, S)$, is the universal $C^*$-algebra generated by partial isometries $A,B,C,D,E,$ $F,G,H$ such that if $P_A, P_B, \ldots $ denote respectively the range projections of $A, B, \ldots$  and $Q_A, Q_B, \ldots$ denote the initial projections of $A,B, \ldots$ then the matrix
 \begin{equation} \begin{bmatrix}
                           P_A & P_B & P_C & P_D \\
                           P_E & P_F & P_G & P_H \\
                           Q_B & Q_A & Q_D & Q_C \\
                           Q_F & Q_E & Q_H & Q_G \\
                         \end{bmatrix}
                         \label{magicunit} \end{equation}
is a magic unitary (all entries are orthogonal projections, the sum of each row/column is equal to 1). The coproduct in $\qlg$ is determined by the condition that the (unitary) matrix
\[ U =\begin{bmatrix}
                           A & B & C & D \\
                           B^* & A^* & D^* & C^* \\
                           E & F & G & H \\
                           F^* & E^* & H^* & G^* \\
                         \end{bmatrix}
                         \]
is a fundamental corepresentation of $\qlg$. In particular the restriction of the coproduct of $\qlg$ to the $C^*$-algebra
generated by the entries of the matrix in \eqref{magicunit} coincides with the coproduct on Wang's $A_4$ -- the universal compact
quantum group acting on 4 points.
\end{tw}

\begin{proof}
The discussion before the formulation of the theorem in conjunction with the results of Section \ref{genth} implies that $\qlg$
is generated by the partial isometries $A,B,C,D$, $E,F,G,H$ and that the matrix in \eqref{magicunit} is a magic unitary. It
remains to show that there are no other relations involved. A direct computation shows that both the unitarity of $U$ and the
conditions in \eqref{alphafree} follow from the fact that the matrix in \eqref{magicunit} is a magic unitary.

It remains to show that the $^*$-homomorphism $\alpha$ obtained as an extension of the formulas for $\alpha(\la_s)$ and
$\alpha(\la_t)$ (and their inverses) given in the  beginning of this section commutes with $ \hat{\Dirac}$ (see the discussion
before Definition \ref{C_hat}). This is essentially a consequence of freeness of $s$ and $t$. Here is a detailed argument: for
each $n \in \bn$ let $\wt{W}(n)=\{\wt{w} = x_1 \cdots x_n: x_1, \ldots, x_n \in S\}$. Note that elements of $\wt{W}_n$ have to be
thought of as formal expressions; there is a natural map $\iota_n$ from $\wt{W}_n$ to $\freeg$, but it is not injective for $n
\geq 2$. The set $W_n$ (the notation introduced earlier) can be identified via $\iota_n^{-1}|_{W_n}$ with a set of these words in
$\wt{W}_n$ which are in reduced form. As usual we write $W_0 = \wt{W}_0 =\{e\}$. For each $z,x \in S$ let $q_{x,z}\in \qlg$ be
such that $\alpha(\la_x) = \sum_{z \in S} \la_{z} \ot q_{x,z}$. Define for each $n \in \bn$, $w \in W_n$, $\wt{w} \in \nnred$
\begin{equation} q_{w, \wt{w}} := \prod_{i=1}^n q_{w_i, \wt{w}_i}\label{qww} \end{equation} and put
$q_{e,e}=1_{\qlg}$. The homomorphic property of $\alpha$ implies that for each $w \in W_n$ we have
\[ \alpha(\la_w) = \sum_{\wt{w} \in \wt{W}_n} \la_{\iota_n(\wt{w})} \ot q_{w, \wt{w}}.\]
 The fact that $\alpha$ commutes with $ \hat{\Dirac} $  is equivalent to the following fact: for all $n \in
\bn, w \in W_n$ and $\gamma \in \freeg$ such that $l(\gamma) < n$ there is \begin{equation} \sum_{\wt{w}\in
\nnred:\,\iota_n(\wt{w}) = \gamma} q_{w, \wt{w}} = 0.\label{0form}\end{equation}  We will now show that the  last statement
holds.

Fix  $n \in \bn, n \geq 2$ and let $w \in W_n, \gamma \in \freeg, l(\gamma) <n$. Then the set $B_{\wt{w}, \gamma}:=\{\wt{w}\in
\nnred:\,\iota_n(\wt{w}) = \gamma\}$ can be split into disjoint subsets ($k=0,\ldots,n-2$) \[B_{\wt{w}, \gamma, k}:=\{\wt{w}\in
\nnred:\,\iota_n(\wt{w}) = \gamma, \; \wt{w}_{k+1} =(\wt{w}_{k+2})^{-1}, \wt{w}_{i+1} \wt{w}_{i+2}\neq e \textup{ for } i <
k\}.\] Further each of the sets $B_{\wt{w}, \gamma, k}$ splits into disjoint (possibly empty) subsets ($\wt{u}\in \knred, \wt{v}
\in \lnred$)
\[B_{\wt{w}, \gamma, k, \wt{u},\wt{v}}:= \{\wt{w} \in B_{\wt{w}, \gamma, k}: \wt{w}_1 \cdots \wt{w}_k = \wt{u}, \;  \wt{w}_{k+3} \cdots \wt{w}_n =
\wt{v}\}\] (the definitions above need to be interpreted in an obvious way when $k=0$ or $k=n-2$). To prove \eqref{0form} it then
suffices to show that for each fixed $k\in \{0, \ldots, n-2\}$, $\wt{u}\in \knred, \wt{v} \in \lnred$ we have
\[ \sum_{\wt{w}\in B_{\wt{w}, \gamma, k, \wt{u},\wt{v}}}
 q_{w, \wt{w}} = 0.\]
Let $u \in W_k, v \in W_l$ and $y,z \in S$ be uniquely determined elements such that $w= u yz v$ ($y \neq z^{-1}$). Then it
follows from the definition in \eqref{qww} that
\[\sum_{\wt{w}\in B_{\wt{w}, \gamma, k, \wt{u},\wt{v}}}
 q_{w, \wt{w}}  = \sum_{x \in S} q_{u, \wt{u}} \, q_{yz, xx^{-1}} \, q_{v, \wt{v}}.\]
It is therefore sufficient to check that if $y,z \in S$ and $y \neq z^{-1}$ then $\sum_{x \in S}   q_{yz, xx^{-1}} =0$. This
follows from a direct computation.

The last statement in the theorem follows once again from a direct computation based on the form of the respective fundamental
corepresentations.
\end{proof}

Similar result holds for free groups of $n$ generators for all $n >2$. In each case denoting free generators by $t_1, \ldots,
t_n$ and putting $S_n=\{t_1, t_1^{-1}, \ldots, t_n, t_n^{-1}\}$ we obtain that $\QISO^+(\mathbb{F}_n, S_n)$ is generated by $2n$
partial isometries such that the corresponding range and initial projections satisfy relations coming from Wang's $A_{2n}$ group.
In the quantum group language one is tempted to say that $A_{2n}$ is a quotient of $\QISO^+(\mathbb{F}_n, S_n)$ (remembering that
$A_{2n}$ is a $C^*$-subalgebra of $\QISO^+(\mathbb{F}_n, S_n)$). In fact the proof of the statement that the subalgebra of
$\QISO^+(\mathbb{F}_n, S_n)$ generated by the entries of the matrix in \eqref{magicunit} is isomorphic with $A_{2n}$ involves the
analysis of the category of corepresentations of $\QISO^+(\mathbb{F}_n, S_n)$. It turns out that  $\QISO^+(\mathbb{F}_n, S_n)$
can be thought of as a `free quantum version' of $\mathbb{T}^n \ltimes D_{2n}$, the isometry group of $\mathbb{T}^n$, and is also
closely connected to the quantum hyperoctahedral group (\cite{hyperoct}). All these facts are proved in \cite{TeoAdam}.

\section{The real structure for the spectral triples on group $C^*$-algebras and the associated quantum isometry groups}

The original notion of a real structure for an (odd) spectral triple $(\Alg, \Hil, \Dirac)$ was introduced in \cite{book}, see
also \cite{Connesgrav}. We will adopt the following definition:

\begin{deft} \label{tildeJ}
For a spectral triple $(\Alg, \Hil, \Dirac)$ a real structure is given by a (possibly unbounded, invertible) closed anti-linear
operator $\wt{J} $ on $\Hil$ such that $ {\tu{Dom}} (\Dirac) \subseteq \tu{Dom} (\wt{J})$, $\wt{J}\,  \Dom(\Dirac) \subseteq \Dom
(\Dirac)$, $\wt{J} $ commutes with $\Dirac$ on $\Dom (\Dirac)$, and the antilinear isometry $J$ obtained from the polar
decomposition of $\wt{J}$ satisfies the usual conditions $J^2 = I, J\Dirac = \Dirac J.$ We say that the quadruple $(\Alg, \Hil,
\Dirac, \wt{J})$ satisfies the first order condition if for all $a, b \in \Alg$ the commutators $ [a, JbJ^{-1}] $ and $ [J a
J^{-1}, [\Dirac, b]]$ vanish.
\end{deft}

Let $\Gamma$ be again a finitely generated discrete group. The triple $(\bc[\Gamma], l^2(\Gamma), \Dirac_{\Gamma})$ has a natural
real structure (Proposition 3.5 in \cite{Conti}) induced by the antiunitary operator $J:l^2(\Gamma) \to l^2(\Gamma)$ defined by
the (antilinear extension of) formula
\begin{equation}\label{realstr}J(\delta_{\gamma}) =  \delta_{\gamma^{-1}}, \gamma \in
\Gamma.\end{equation}

We define $ \wt{J} = J $ (this is compatible with the notation in Definition \ref{tildeJ}). As can be deduced from Proposition
3.15 of \cite{Conti} the quadruple $(\bc[\Gamma], l^2(\Gamma), \Dirac_{\Gamma}, \wt{J})$ practically never satisfies the first
order condition. The following weakening of that condition was considered in \cite{DLPS} and \cite{Dabrowski}; it will be
sufficient for our purposes.

\begin{deft}
Let $(\Alg, \Hil, \Dirac, \wt{J})$ be a spectral triple with a real structure. We say that $(\Alg, \Hil, \Dirac, \wt{J})$
satisfies the first order condition modulo compacts if for all $a, b \in \Alg$ the commutators $ [a, JbJ^{-1}] $ and $ [J a
J^{-1}, [\Dirac, b]]$ are compact.
\end{deft}

It is immediate that if $\Gamma$ is finite, then $(\bc[\Gamma], l^2(\Gamma), \Dirac_{\Gamma}, J)$ satisfies the first order
condition modulo compacts in a trivial way. We also have the following easy fact:

\begin{lem} \label{modcomp}
Let $n \in \bn$ and let $\Gamma$ denote either the free group on $n$ generators or the free abelian group on $n$ generators, in
both cases with the standard symmetric generating sets. Then $(\bc[\Gamma], l^2(\Gamma), \Dirac_{\Gamma}, J)$ satisfies the first
order condition modulo compacts.
\end{lem}

\begin{proof}
It is enough to show that for all $g,h \in \Gamma$ the operators $[\lambda_g, J\lambda_hJ^{-1}] $ and $ [J \lambda_g J^{-1},
[\Dirac, \lambda_h]]$ are compact. It is easy to check that $J\lambda_hJ^{-1} = \rho_{h}$, where $\rho$ denotes the right regular
representation of $\Gamma$. Therefore the first commutator above always vanishes. It thus suffices to show that the operator
$T_{g,h}:=[\rho_{g^{-1}}, [\Dirac, \lambda_h]]$ is compact for arbitrary $g,h \in \Gamma$. Straightforward computation yields \be
\label{Toper} T_{g,h} \delta_a = (l(ha) - l(a) - l(hag) +l(ag)) \, \delta_{hag}, \;\;\;a \in \Gamma.\ee Now in both cases
considered in the lemma ($\Gamma= \mathbb{F}_n$ and $\Gamma= \bz^n$) it is easy to check that for fixed choice of $g$ and $h$
there are only finitely many elements $a\in \Gamma$ such that the expression in \eqref{Toper} is non-zero. Hence the operator
$T_{g,h}$ is actually even finite rank.
\end{proof}

\subsection*{Quantum groups of orientation and real structure preserving isometries}

Given a $C^*$-algebra $\alg,$ define the antilinear map $\wt{J_{\alg}}:{\alg} \to {\alg}$ by $\wt{J_{\alg}}(a) = a^*$.

Recall the notion of the quantum group of orientation and real structure preserving isometries of a spectral triple introduced in
\cite{Deb_real}.

\begin{deft}
Suppose that the spectral triple  $(\Alg,\Hil,\Dirac)$ is equipped with a real structure given by $\wt{J}$. We say that a quantum
family of orientation preserving isometries $({\mathcal S},U)$ preserves the real structure if the following equality holds on
${\Dom} (\Dirac)$: \be \label{realstrcomm} ( \wt{J} \ot \wt{J}_{\mathcal S} )  U = U  \wt{J}.\ee In case the $C^*$-algebra
$\mathcal{S}$ has a coproduct $\Delta$ such that $({\mathcal S}, \Delta)$ is a compact quantum group and $U$ is a unitary
corepresentation of $ ({\mathcal S},\Delta)$ on $\Hil$ we say that $({\mathcal S},\Delta)$ acts by orientation and real structure
preserving isometries on the spectral triple.
\end{deft}

One is thus naturally led to consider the  categories $ {\bf Q_{{\rm real} }} ( \Alg, \Hil, \Dirac, \wt{J} ) $ and \\$ {\bf
Q^{\prime}_{\rm real} } ( \Alg, \Hil, \Dirac, \wt{J} )$ analogous to $ {\bf Q} ( \Alg, \Hil, \Dirac ) $ and $ {\bf Q^{\prime}}(
\Alg, \Hil, \Dirac )$, respectively. The following theorem is proved in \cite{Deb_real}.

\begin{tw}
For any spectral triple with a real structure $(\Alg, \Hil, \Dirac, \wt{J})$ the category ${\bf Q_{{\rm real} }} ( \Alg, \Hil,
\Dirac, \wt{J} ) $ of quantum families of orientation and real structure preserving isometries has a universal (initial) object,
to be denoted by $(\wt{{\QISO}^+_{{\rm real}}} ( \Alg, \Hil, \Dirac, \wt{J} ), U_0 )  $ (or, when the context is clear, by
$\wt{{\QISO}^+_{{\rm real}}} ( \Dirac )$.)
 The $C^*$-algebra $\wt{{\QISO}^+_{{\rm real}}} (\Dirac) $ has a coproduct $ \Delta $ such that $ (\wt{{\QISO}^+_{{\rm real}}} ( \Dirac), \Delta) $ is a compact quantum group and $ ( \wt{{\QISO}^+_{{\rm real}}} ( \Dirac ), U_0 ) $ is a universal object in the category $ {\bf Q^{\prime}_{\rm real} } ( \Alg, \Hil, \Dirac, \wt{J} ) .$  The corepresentation $ U_0 $ is faithful.
 \end{tw}

The need to consider `nondegenerate' actions motivates another definition (we keep the notation of the last theorem):

\begin{deft}
The Woronowicz $C^*$-subalgebra of $\wt{{\QISO}^+_{{\rm real}}} ( \Dirac ) $ generated by elements of the form $ \left\langle
\alpha_{U_0} ( a ) ( \eta \ot 1 ), \eta^{\prime} \ot 1 \right\rangle_{\wt{{\QISO}^+_{{\rm real}}}(\Dirac)} $ where $ \eta,
\eta^{\prime} \in \Hil$,   $a \in  \Alg $ and  $ \left\langle \cdot  , \cdot \right\rangle_{\wt{{\QISO}^+_{{\rm real}}}(\Dirac)}
$ denotes the $ \wt{{\QISO}^+_{{\rm real}}}(\Dirac) $ valued inner product of the Hilbert module $ \Hil \ot \wt{{\QISO}^+_{{\rm
real}}}(\Dirac)$, will be called the quantum group of orientation and real structure preserving isometries of the given spectral
triple with a real structure. We will denote it ${\QISO}^+_{{\rm real}}(\Alg, \Hil, \Dirac, \wt{J})$, or, when the context is
clear, by $ {\QISO}^+_{{\rm real}}(\Dirac).$
\end{deft}

 It follows from the above definition that $ {\QISO}^+_{{\rm real}}( \Dirac ) $ is a quantum subgroup of ${\QISO}^+( \Dirac ) $ whenever the latter exists, since the former is the universal object in a subcategory of the category
for which the latter is universal (if it exists). It is easily seen that $ {\QISO}^+_{{\rm real}}( \Dirac ) $ is the quotient of
${\QISO}^+( \Dirac )$ by the commutation relations arising from the condition \eqref{realstrcomm}.

 As discussed above, where  $(\Alg,\Hil,\Dirac) = (\bc[\Gamma], l^2(\Gamma), \Dirac_{\Gamma})$, the real structure is determined
by the antiunitary $J$ defined by the formula \eqref{realstr}. The computation of the $\QISO^+_{\tu{real}}(\bc[\Gamma],
l^2(\Gamma), \Dirac_{\Gamma})$, denoted further by $\QISO^+_{\tu{real}}(\Gamma,S)$ is facilitated by the following lemma:

\blmma \label{groupalgebra_S3_realstructure_lemma} Let the (orientation preserving, isometric) action of a compact quantum group
$({\mathcal S}, \Delta)$ on $(\bc[\Gamma], l^2(\Gamma), \Dirac_{\Gamma})$ be given by a unitary $U$ satisfying
\eqref{unitaction1}-\eqref{unitaction2} (with $\wt{\QISO^+}(\Gamma, S)$ replaced by ${\mathcal S}$), where the coefficients
$[q_{\gamma, \gamma'}]_{\gamma, \gamma' \in \Gamma }$ in ${\mathcal S}$ are determined the formula \eqref{alphag} for the adjoint
action $\alpha_U:=\tu{Ad} U$ of ${\mathcal S}$ on $C^*_r(\Gamma)$ and $ q $ is as in \eqref{U_0_xi}. Then $U \wt{J} = ( J \ot
\wt{J}_{{\mathcal S}} ) U$ if and only if $q$ is selfadjoint and
\begin{equation}   q_{\gamma', \gamma} q  = q q_{\gamma', \gamma}   \label{inverses}\end{equation}
for all $ \gamma, \gamma'$ in $ \Gamma $ such that $ l ( \gamma ) = l ( \gamma'). $
\elmma

\begin{proof}
Observe first that $U \wt{J}(\delta_e )  = ( J \ot \wt{J_{\mathcal S}} ) U ( \delta_e ) $ implies that $ \delta_e \ot q = \delta_e \ot q^* $, so that $ q $ is self-adjoint. Let then $\gamma \in \Gamma$ and compute:
\begin{align*} U \wt{J} ( \delta_\gamma ) &=  U (  \delta_{\gamma^{- 1}} ) =  \widetilde{U} ( \delta_{\gamma^{- 1}} \otimes 1 )           =  \widetilde{U} ( \lambda_{\gamma^{- 1}} \otimes 1 ) {\widetilde{U}}^* \widetilde{U} ( \delta_e \otimes 1 )
\\& =  \alpha ( \lambda_{\gamma^{- 1}} ) \widetilde{U} ( \delta_e \otimes 1 )
 = \alpha ( \lambda_{\gamma^{- 1}} ) ( \delta_e \otimes q )\\&= \left( \sum_{\gamma^{\prime}: ~ l ( \gamma^{\prime} ) = l ( \gamma ) }  \lambda_{\gamma^{\prime}} \otimes q_{\gamma^{\prime}, \gamma^{- 1}} ~ \right) ( \delta_e \otimes q )
 = \sum_{\gamma^{\prime}: ~ l ( \gamma^{\prime} ) = l ( \gamma ) } \delta_{\gamma^{\prime}} \otimes q_{\gamma^{\prime}, \gamma^{- 1}} q .\end{align*}
On the other hand
\begin{align*}  ( \wt{J} \otimes \wt{J_{{\mathcal S}}}  )\, & U ( \delta_\gamma ) =
   ( \wt{J} \otimes \wt{J_{{\mathcal S}}} ) \alpha ( \lambda_\gamma ) ( \delta_e \otimes q )
  \\&=  ( J \otimes \wt{J_{{\mathcal S}}} ) \left( \sum_{\gamma^{\prime}: ~ l ( \gamma^{\prime} )= l ( \gamma ) } \lambda_{\gamma^{\prime}} \otimes q_{\gamma^{\prime},\gamma} ~ \right) ( \delta_e \otimes q )
  \\&= ( J \otimes \wt{J_{{\mathcal S}}} ) \left( \sum_{\gamma^{\prime}: ~ l ( \gamma^{\prime} ) = l ( \gamma ) } \delta_{\gamma^{\prime}} \otimes q_{\gamma^{\prime},\gamma} q ~ \right)
  = \sum_{\gamma^{\prime}: ~ l ( \gamma^{\prime} ) = l ( \gamma ) } \delta_{{\gamma^{\prime}}^{- 1}} \otimes {( q_{\gamma^{\prime},\gamma} q )}^*. \end{align*}

Thus, $U \wt{J} = ( J \ot \wt{J_{{\mathcal S}}} ) U$ can hold if and only if $q$ is self-adjoint and  for all $ \gamma,
\gamma^{\prime} $ in $ \Gamma $ such that $ l ( \gamma ) = l ( \gamma^{\prime} ), ~  q_{{\gamma^{\prime}}^{- 1},\gamma} q = (
q_{\gamma^{\prime}, \gamma^{- 1}} q )^*.$ As $ q $ is a self-adjoint unitary, we have  $ q_{{\gamma^{\prime}}^{- 1}, \gamma} q =
q {q_{\gamma^{\prime}, \gamma^{- 1}}}^*  $ for all such $ \gamma, \gamma^{\prime}. $
 But by \eqref{alpha_star_hom_cond}, this is equivalent to the condition that
 $ q_{{\gamma^{\prime}}^{- 1}, \gamma} q  = q q_{{\gamma^{\prime}}^{- 1}, \gamma} $ for all $ \gamma, \gamma^{\prime} $ in $ \Gamma $
 such that $ l ( \gamma ) = l ( \gamma^{\prime} ) $. As we can replace $\gamma'$ by $\gamma'^{-1}$, the lemma follows.
\end{proof}

The above lemma leads to a simple description of the quantum groups  of orientation and real structure preserving isometries in
terms of the quantum groups of orientation preserving isometries. It is formulated in the next theorem.

\begin{tw} \label{realisom}
Let $\Gamma$ be a finitely generated discrete group with a symmetric generating set $S$. Assume that the associated spectral
triple with a real structure $(\bc[\Gamma], l^2(\Gamma), \Dirac_{\Gamma}, J)$ satisfies the first order condition modulo
compacts. Then ${\QISO}^+_{{\rm real}} (\Gamma, S) \cong {\QISO}^+ (\Gamma, S) $ and  $ \wt{{\QISO}^+_{{\rm real}}} ( \Gamma, S)
\cong {\QISO}^+ (\Gamma, S) \ot \bc^2$ (as a $C^*$-algebra).
\end{tw}

\begin{proof}
The first statement follows from the fact that the conditions in Lemma \ref{groupalgebra_S3_realstructure_lemma} are trivially satisfied for $q=1$.

Let $\mathcal{S}_1$ denote the $C^*$-algebra generated by ${\QISO}^+ (\Gamma, S)$ and a self-adjoint unitary $q$ commuting with
${\QISO}^+ (\Gamma, S)$. As a selfadjoint unitary is the difference of two complementary projections, it is easy to see that
$\mathcal{S}_1 \cong {\QISO}^+ (\Gamma, S) \ot \bc^2$. It follows from Lemma \ref{groupalgebra_S3_realstructure_lemma} that $
{\mathcal S_1} $ is a subobject of $\wt{{\QISO}^+}_{{\rm real}} (\Gamma,S) $ in the category $ {\bf Q^{\prime}_{\rm real} }
(\bc[\Gamma], l^2(\Gamma), \Dirac_{\Gamma}, J )$ (note that $\QISO^+_{{\rm real}} (\Gamma,S)$ is generated by the elements
$q_{s,t}$ where $s,t \in S$).  Conversely, as the relation \eqref{inverses} has to be satisfied by the appropriate elements of
$\wt{{\QISO}^+}_{{\rm real}} ( \Gamma,S) $,  $\wt{{\QISO}^+}_{{\rm real}} ( \Gamma, S ), $  is a subobject of $ {\mathcal S_1} $
in the same category. The theorem follows by the universality of  $\wt{{\QISO}^+}_{{\rm real}} ( \Gamma,S). $
\end{proof}

As follows from Lemma \ref{modcomp} and the discussion before it, in all the examples considered so far the assumptions of the
above theorem are satisfied. Thus the corresponding quantum groups  of orientation and real structure preserving isometries can
be read out directly from Theorem \ref{realisom} and computations in Sections 4-6.

\section{Construction of the Laplacian}

From the point of view of quantum isometry groups (and also for general noncommutative geometry interest) it is important to know
whether the  spectral triples $(\bc[\Gamma], l^2(\Gamma), \Dirac_{\Gamma})$  on group $C^*$-algebras are `admissible' in the
sense of \cite{Deb}; in other words whether we have a `good' Laplacian operator.

Following \cite{Deb} we describe below briefly how one constructs the Laplacian given a $\theta$-summable spectral triple $(\Alg,
\Hil, \Dirac).$ The idea is based on first completing $\Alg$ with respect to the scalar products resulting from a state
$\omega_\Dirac$ given by the formula
\[ \omega_\Dirac (a)= \textup{LIM}_{t\to 0^+} \frac{\Tr (ae^{-t\Dirac^2})}{\Tr(e^{-t\Dirac^2})}, \;\;\; a \in \Alg,\]
where $\LIM$ denotes some generalised limit (see \cite{Deb}, \cite{book}). This leads to a Hilbert space $\Hil_0$ (playing the
role of the classical $L^2$-space with respect to the Riemannian volume form --  recall that $\Hil$ corresponds to the Hilbert
space of spinors!). Observe that as $\exp(-t\Dirac^2)$ is a scalar multiple of a faithful density matrix, the natural map from
$\Alg$ to $\Hil_0$ is injective. Analogous completion, this time with respect to the scalar product involving the commutators of
elements in $\Alg$ with the Dirac operator, leads to $\Hil_1$, the Hilbert space playing the role of the Hilbert space of
$1$-forms on a classical Riemannian manifold. The operation of taking the commutator with $\Dirac$ provides a natural linear
densely defined operator from $\Hil_0$ to $\Hil_1$ denoted by $d$. If $d$ is closable, denote its closure by $\overline{d}$. The
Laplacian is then  a selfadjoint operator on $\Hil_0$ given by
\[ \Lap = \overline{d}^* \overline{d}.\]
We call the spectral triple $(\Alg, \Hil, \Dirac)$ admissible if $d$ is closable, the domain of $ \Lap $ contains $\Alg$ (viewed
as a subspace of $\Hil_0$), $\Alg$ is left invariant by $\Lap$, $\Lap$ has compact resolvent, the kernel of $ \Lap $ is
one-dimensional and the span of its eigenvectors is a norm dense subspace of $\Alg.$  Note that to establish compactness of the
resolvent it suffices to check that the eigenspaces of $\Lap$ are finite-dimensional and eigenvalues increase to infinity (if
$\Alg$ is infinite dimensional). If $(\Alg, \Hil, \Dirac)$ is admissible, we can also view $ \Lap $ as a linear operator from
$\Alg$ to $\Alg$. If a spectral triple is admissible, D.\,Goswami showed in \cite{Deb} that one can associate to it a canonical
quantum isometry group, denoted ${\QISO}^{{\mathcal L}} (\alg )$. We refer to the article \cite{Deb} for the precise definitions;
naively one can think of ${\QISO}^{{\mathcal L}} (\alg )$ as the universal compact quantum group acting on $\alg$ in the manner
preserving the eigenspaces of the Laplacian $\mathcal{L}$.

We will now describe a method of investigating the existence and admissibility of the Laplacian of a given triple without
describing explicitly Hilbert spaces $\Hil_0$ and $\Hil_1$. It is based on the following observation: suppose $(\Alg, \Hil,
\Dirac)$ is admissible so that $ \Lap $ exists, and consider the associated operator on $\Alg$ (denoted by the same letter). Then
for all $a, b \in \Alg$ we must have
\begin{equation} \label{Lapscal} \langle \Lap(a), b \rangle_{\Hil_0} = \langle [\Dirac,a], [\Dirac,b]\rangle_{\Hil_1},\end{equation}
moreover this condition determines $\Lap$ uniquely. This observation leads to the determination of a potential form of the
Laplacian; once this is known one can attempt to show that it has all the properties required for admissibility.

We are ready to apply the above discussion for spectral triples of the type $(\bc[\Gamma], l^2(\Gamma), \Dirac_{\Gamma})$, where
$\Gamma$ is a finitely generated discrete group with a generating set $S$. As the formula \eqref{Lapscal} is clearly linear in
$b$ and conjugate linear in $a$ it is enough to check what it says for $a=\la_{\gamma}$, $b= \la_{\gamma'}$ for some $\gamma,
\gamma' \in \Gamma$. Suppose that $\Lap(\la_{\gamma}) = \sum_{\gamma'' \in \Gamma} c_{\gamma, \gamma''} \la_{\gamma''}$, with
only finitely many complex coefficients $c_{\gamma, \gamma''}$ being non-zero. Then the left hand side of \eqref{Lapscal} reads:
\begin{align*}
\LIM_{t \to 0^+} \sum_{\kappa \in \Gamma} \langle \Lap(\la_{\gamma}) \delta_{\kappa}, \la_{\gamma'} e^{-t l(\kappa)^2}
\delta_{\kappa}\rangle  &= \LIM_{t \to 0^+} \sum_{\kappa, \gamma'' \in \Gamma} \langle c_{\gamma, \gamma''} \delta_{\gamma''
\kappa}, e^{- t l(\kappa)^2}  \delta_{\gamma'\kappa} \rangle
\\& = \LIM_{t \to 0^+}  c_{\gamma, \gamma'} \sum_{\kappa \in \Gamma} e^{-t l(\kappa)^2}
\end{align*}
and the right hand side reads
\begin{align*}
\LIM_{t \to 0^+} \sum_{\kappa \in \Gamma}  \langle [\Dirac,\la_{\gamma}] &\delta_{\kappa}, [\Dirac, \la_{\gamma'}] e^{-t
l(\kappa)^2} \delta_{\kappa}\rangle  \\ &= \LIM_{t \to 0^+} \sum_{\kappa \in \Gamma} \langle (l(\gamma\kappa) - l(\kappa))
\delta_{\gamma\kappa}, (l(\gamma'\kappa) - l(\kappa)) e^{-tl(\kappa)^2} \delta_{\gamma'\kappa}\rangle
\\&= \LIM_{t \to 0^+} \delta_{\gamma}^{\gamma'} \; \sum_{\kappa \in \Gamma} |l(\gamma\kappa) - l(\kappa)|^2 e^{-tl(\kappa)^2}.
\end{align*}
This shows that if $\Lap$ associated with $(\bc[\Gamma], l^2(\Gamma), \Dirac_{\Gamma})$ exists and is admissible, then it is
automatically `diagonal', i.e.\ for each $\gamma \in \Gamma$ there exists $c_{\gamma} \in \bc$ such that
\begin{equation}\label{Lapform}\Lap (\la_{\gamma}) = c_{\gamma} \la_{\gamma};\end{equation}
 moreover
\begin{equation} \label{cfor} c_{\gamma} = \LIM_{t \to 0^+} \frac{\sum_{\kappa \in \Gamma} |l(\gamma\kappa) - l(\kappa)|^2 e^{-tl(\kappa)^2}}
{\sum_{\kappa \in \Gamma}  e^{-tl(\kappa)^2}}.\end{equation} Suppose first that $\Gamma$ is infinite. For a fixed $\gamma\in
\Gamma$ and each $t>0, \kappa \in \Gamma$ we have
\[ 0 \leq |l(\gamma\kappa) - l(\kappa)|^2 e^{-tl(\kappa)^2} \leq l(\gamma)^2 e^{-tl(\kappa)^2},\]
so that the value of the expression after the limit in \eqref{cfor} is always between $0$ and $l(\gamma)^2$; in particular
$c_e=0$. Properties of generalised limit imply that $c_{\gamma}$ is always defined and $c_{\gamma} \in [0, l(\gamma)^2]$. Thus
(first fixing some generalised limit $\LIM$) we can rigorously deduce that the operator $\Lap$ defined by the formula
\eqref{Lapform} and interpreted as a densely defined operator on $\Hil_0$ is indeed (due to formula \eqref{Lapscal}) an extension
of the densely defined operator $d^*d$. The admissibility of $\Lap$ can be then checked directly by inspecting the properties of
eigenspaces and eigenvalues of $\Lap$. More precisely, observe that if we could show that the (standard) limit in \eqref{cfor}
exists, $c_{\gamma}=c_{\gamma'}$ if and only if $l(\gamma) = l(\gamma')$ and moreover $c_{\gamma}$ converges to infinity as
$l(\gamma)$ tends to infinity, then it would follow that $\Lap$ is admissible. To obtain the above statements one needs to
analyse for a fixed $\gamma \in \Gamma$ the behaviour of the expression
\begin{equation} \label{cfort} c_{t,\gamma} = \frac{\sum_{\kappa \in \Gamma} |l(\gamma\kappa) - l(\kappa)|^2 e^{-tl(\kappa)^2}}
{\sum_{\kappa \in \Gamma}  e^{-tl(\kappa)^2}},\end{equation}
 as $t\to 0^+$. Note that we can rewrite the expression above as
\begin{equation} \label{cfort2} c_{t,\gamma} = \frac{\sum_{n=0}^{\infty} e^{-tn^2} (\sum_{\kappa \in W_n}  |l(\gamma\kappa) - l(\kappa)|^2) }
{\sum_{n=0}^{\infty}  e^{-tn^2} \card{W_n}} .\end{equation} Thus in concrete examples we will need to understand the ratio
\begin{equation} \label{ratio}r_{n, \gamma}:=\sum_{\kappa \in W_n}  |l(\gamma\kappa) - l(\kappa)|^2 (\card{W_n})^{-1}.\end{equation}

Note finally that if $\Gamma$ is finite, then the (standard) limit in \eqref{cfor} exists and we have
\begin{equation} \label{Lapfin}c_{\gamma}
= \frac{\sum_{\kappa \in \Gamma} |l(\gamma\kappa) - l(\kappa)|^2 } {\tu{card} \,\Gamma}.\end{equation}
 The admissibility of
$\Lap$ does not form a problem in that case, as all Hilbert spaces considered are finite-dimensional. The Laplacian also
satisfies the connectedness condition (one-dimensionality of the kernel), as $c_{\gamma}= 0 $ if and only if $\gamma=e$. For
concrete computations one may still need to understand the dependence between the value of $c_{\gamma}$ and the length of
$\gamma$.

\subsection*{ Computation of  ${\QISO}^{{\mathcal L}} ( C^* ( S_3 ) )$}
A direct computation of the expressions in \eqref{Lapfin} for $\Gamma=S_3$ and both sets of generators considered in Section
\ref{S3Section} shows that in each case $c_{\gamma}=c_{\gamma'}$ if and only if $l(\gamma) = l(\gamma')$ ($\gamma, \gamma' \in
S_3$). Thus, with respect to both $ S $ and  $ S^{\prime}, ~ {\QISO}^{{\mathcal L}} ( C^* ( S_3 ) ) $ coincides with $ {\QISO}^+ ( C^* ( S_3 ) )$ computed in Section \ref{S3Section}.

\subsection*{ Computation of ~ ${\QISO}^{{\mathcal L}} ( C^*_r ( \freeg ) )$}

We need to investigate for a fixed $\gamma \in \freeg$ the behaviour of the ratio $r_{n, \gamma}$ defined in \eqref{ratio}. We fix
 $\gamma \in \freeg$ and let $m = l(\gamma)$. We first note that $\sum_{\kappa \in W_n} |l(\gamma\kappa) - l(\kappa)|^2$ depends
only on the length of $\gamma$ (and not on its actual form). Moreover for $3$ out of $4$ words in $W_n$ (precisely speaking for
all these $\kappa\in W_n$ such that $\kappa_1 \neq (\gamma_{m})^{-1}),$ $l(\gamma \kappa) = l(\gamma) + l(\kappa)$. Hence
we have $r_{n, \gamma} \in [\frac{3}{4}l(\gamma)^2,l(\gamma)^2]$. Moreover if $n, n' \geq m$ it is easy to observe that $r_{n,
\gamma} = r_{n', \gamma}:= R_{l(\gamma)}$ (only first $m$ letters of a word $\kappa \in W_n$ influence possible cancellations in
$\gamma \kappa$). This implies that (recall \eqref{cfort2})
\[ c_{t, \gamma} = \frac{\sum_{n=0}^{\infty} e^{-tn^2} r_{n, \gamma} \card (W_n) }
{\sum_{n=0}^{\infty}  e^{-tn^2} \card(W_n)} = R_{l(\gamma)} + \frac{\sum_{n=0}^{m-1} e^{-tn^2} (r_{n, \gamma} -
R_{l(\gamma)})}{\sum_{n=0}^{\infty}  e^{-tn^2} \card(W_n)}.\] Therefore,
\[ c_{\gamma} = \lim_{t \to 0^+} c_{t, \gamma} = R_{l(\gamma)} \in \left[\frac{3}{4}l(\gamma)^2,l(\gamma)^2\right].\]
This suffices to observe that if $m\neq m'$ then $R_m \neq R_m'$ and $R_m\nearrow \infty$ as $m$ tends to $\infty$, so the spectral triple is
admissible.

Note that the actual values of $R_{m}$ can be computed combinatorially by looking at the possibilities for cancellations. For
example $R_1=1$, $R_2 = \frac{13}{16} 2^2 $, $R_3 = \frac{200}{256} 3^2$, etc..

The discussion above remains valid for free group on arbitrarily many generators; the factor $\frac{3}{4}$ featuring in the
estimate above will be in general replaced by $\frac{2n-1}{2n}$, where $n$ is the number of generators. As in each case we have
$c_{\gamma}=c_{\gamma'}$ if and only if $l(\gamma) = l(\gamma')$ ($\gamma, \gamma' \in \mathbb{F}_n$), $ {\QISO}^{{\mathcal L}} (
C^*_r ( \freeg ) ) $ will coincide with $ {\QISO}^+_\Dirac ( C^*_r ( \freeg ) )   $ computed in Section \ref{freegroupqiso}.

\vspace*{0.2cm} \noindent \textbf{Acknowledgment.}
The work on this paper was started during the visit of the
 first named author to the Lancaster University in the summer 2009, which was made possible thanks to the support of the UKIERI
project Quantum Probability, Noncommutative Geometry and Quantum Information. We would like to thank Teo Banica for useful
comments on the first draft of the paper.

\end{document}